\documentclass[twoside,a4paper,12pt,leqno]{amsart}
\usepackage{amssymb,graphicx,mathrsfs}
\usepackage[PostScript=dvips,nohug]{diagrams}
\AtEndDocument{\vfill\eject\batchmode}
\makeatletter
\theoremstyle{plain}

\newtheorem{thm}{Theorem}[section]
\newtheorem{prop}[thm]{Proposition}
\newtheorem{cor}[thm]{Corollary}
\newtheorem{lemma}[thm]{Lemma}

\theoremstyle{definition}
\newtheorem{defn}[thm]{Definition}

\theoremstyle{remark}
\newtheorem{rem}[thm]{Remark}
\newtheorem{rems}[thm]{Remarks}

\newtheorem{example}{Example}[section]

\newcommand{\pfin}[1]{Proof \textup(see~\cite{#1}\textup)}
\numberwithin{equation}{section}

\newcounter{numl}

\newcommand{\labelnuml}{\textup{(\arabic{numl})}}
\newenvironment{numlist}{\begin{list}{\labelnuml}%
{\usecounter{numl}\setlength{\leftmargin}{0pt}%
\advance\@listdepth1\relax%
\setlength{\itemindent}{2\parindent}%
\setlength{\itemsep}{0pt}
\def\makelabel ##1{\hss \llap {\upshape ##1}}}}{\advance\@listdepth-1\relax%
\end{list}}
\newenvironment{bulletlist}{\begin{list}{\labelitemi}%
{\setlength{\leftmargin}{\parindent}%
\advance\@listdepth1\relax%
\def\makelabel ##1{\hss \llap {\upshape ##1}}}}{\advance\@listdepth-1\relax%
\end{list}}
\newcommand{\acknowledge}{\subsection*{Acknowledgements}}
\newcommand{\thismonth}{\ifcase\month\or
  January\or February\or March\or April\or May\or June\or
  July\or August\or September\or October\or November\or December\fi
  \space\number\year}
\renewcommand{\@secnumfont}{\relax}
\newcommand{\low}{\@ifnextchar^{}{^{\vphantom x}}}
\newcommand{\high}{\@ifnextchar_{}{_{\vphantom I}}}
\DeclareSymbolFont{script}{U}{eus}{m}{n}
\DeclareMathSymbol{\EuWedge}{0}{script}{"5E}
\DeclareMathAlphabet{\mathrmsl}{OT1}{cmr}{m}{sl}
\newcommand{\rssymb}[2]{\newcommand{#1}{{\mathrmsl{#2}}}}
\newcommand{\calsymb}[2]{\newcommand{#1}{{\mathcal{#2}}}}
\newcommand{\bbsymb}[2]{\newcommand{#1}{{\mathbb{#2}}}}
\newcommand{\liealg}[2]{\newcommand{#1}{{\mathfrak{#2}}\low}}
\newcommand{\liealr}[2]{\renewcommand{#1}{{\mathfrak{#2}}\low}}
\newcommand{\lieoper}[2]{\newcommand{#1}{\mathop
  {\mathfrak{#2}\null}\nolimits}}
\newcommand{\oper}[3][n]{\newcommand{#2}{\mathop
  {\mathrm{#3}\null}\ifx n#1\nolimits\else\limits\fi}}
\newcommand{\rsoper}[3][n]{\newcommand{#2}{\mathop
  {\mathrmsl{#3}\null}\ifx n#1\nolimits\else\limits\fi}}
\bbsymb\C{C} \bbsymb\F{F} \bbsymb\HQ{H}\bbsymb\N{N} \bbsymb\Q{Q}
\bbsymb\R{R} \bbsymb\U{U} \bbsymb\V{V} \bbsymb\W{W} \bbsymb\Z{Z}
\calsymb\cA{A} \calsymb\cB{B} \calsymb\cC{C} \calsymb\cD{D} \calsymb\cE{E}
\calsymb\cF{F} \calsymb\cG{G} \calsymb\cH{H} \calsymb\cI{I} \calsymb\cJ{J}
\calsymb\cK{K} \calsymb\cL{L} \calsymb\cM{M} \calsymb\cN{N} \calsymb\cO{O}
\calsymb\cP{P} \calsymb\cQ{Q} \calsymb\cR{R} \calsymb\cS{S} \calsymb\cT{T}
\calsymb\cU{U} \calsymb\cV{V} \calsymb\cW{W} \calsymb\cX{X} \calsymb\cY{Y}
\calsymb\cZ{Z}

 \newcommand{\Gam}{{\mathrmsl\Gamma}}
\newcommand{\lam}{\lambda}\newcommand{\Lam}{{\mathrmsl\Lambda}}
\renewcommand{\geq}{\geqslant} \renewcommand{\leq}{\leqslant}
\rsoper\End{End} \rsoper\Hom{Hom}                
\rsoper\Sym{Sym} \rsoper\Skew{Skew}
\rsoper\Aut{Aut} \rsoper\SAut{SAut}              
\rsoper\GL{GL}\rsoper\SL{SL}\rsoper\PGL{PGL}\rsoper\PSL{PSL}\rsoper\Symp{Sp}
\rsoper\CO{CO}\rsoper\On{O} \rsoper\SO{SO}  \rsoper\Pin{Pin}\rsoper\Spin{Spin}
\rsoper\CU{CU}\rsoper\Un{U} \rsoper\SU{SU}
\rsoper\Diff{Diff} \rsoper\SDiff{SDiff}
\lieoper\der{der}                                
\lieoper\gl{gl} \lieoper\sgl{sl}\lieoper\symp{sp}
\lieoper\co{co} \lieoper\so{so} \lieoper\spin{spin}
\lieoper\cu{cu} \lieoper\un{u}  \lieoper\su{su}
\rsoper\Vect{Vect} \rsoper\Ham{Ham}
\rsoper\Stab{Stab} \lieoper\stab{stab} 
\oper\real{Re}  
\newcommand{\ip}[1]{\langle#1\rangle}

\newcommand{\norm}[2][]{|\mkern-2mu|#2|\mkern-2mu|
  _{\lower1pt\hbox{${}_{#1}$}}}
\newcommand{\Norm}[2][]{\bigl|\mkern-3mu\bigr|#2\bigr|\mkern-3mu\bigr|
  _{\lower1pt\hbox{${}_{#1}$}}}

\newcommand{\restr}[1]{|_{#1}\low}

\newcommand{\transp}{^{\scriptscriptstyle\mathrm T\!}}
\newcommand{\act}{\mathinner\cdot}          
\newcommand{\dsum}{\oplus}                  
\newcommand{\Dsum}{\bigoplus}               
\newcommand{\tens}{\otimes}                 
\newcommand{\Wedge}{\EuWedge}               
\newcommand{\idealin}{\trianglelefteq}      
\newcommand{\subnormal}{\ltimes}            
\newcommand{\into}{\hookrightarrow}         
\newcommand{\Gr}{\mathrmsl{Gr}}             
\newcommand{\Proj}{\mathrmsl{P}}            
\newcommand{\RP}[1]{\R\Proj^{#1}}           
\newcommand{\st}{\mathrel{|}}               
\newcommand{\ie}{\textit{i.e.}}             
\rsoper\dimn{dim}                           
\rsoper\kernel{ker}\rsoper\image{im}        
\rsoper\alt{alt}   \rsoper\sym{sym}         
\rsoper\Ad{Ad}     \rsoper\ad{ad}           
\rsoper\CoAd{CoAd} \rsoper\coad{coad}       
\rsoper\trace{tr}  \rsoper\trfree{tf}       
\rsoper\detm{det}                           
\rsoper\Vol{Vol}                            
\rsoper\divg{div}                           
\rssymb\iden{id}                            
\rssymb\vol{vol}                            
\makeatother
\lieoper{\nil}{nil}
\rsoper{\spn}{span}
\rsoper{\gr}{gr}
\rsoper{\rank}{rank}
\newcommand{\empt}{\varnothing}
\newcommand{\sub}{\subseteq}
\liealr{\a}{a} \liealr{\b}{b} \liealr{\c}{c}
\liealg{\f}{f} \liealg{\g}{g} \liealg{\h}{h} \liealg{\gp}{g'}
\liealr{\k}{k} \liealr{\l}{l} \liealg{\m}{m} \liealg{\n}{n}
\liealg{\p}{p} \liealg{\q}{q} \liealr{\r}{r} \liealg{\pp}{p'} \liealg{\rp}{r'}
\liealg{\s}{s} \liealr{\t}{t} \liealr{\u}{u} \liealg{\z}{z}
\liealg{\tu}{\widetilde u}
\liealg{\hp}{\widehat p} \liealg{\hpp}{\widehat p'{}}
\liealg{\qq}{q'} \liealg{\hq}{\widehat q}
\liealg{\pb}{b} \liealg{\hpb}{\widehat b}
\liealg{\pc}{c} \liealg{\hpc}{\widehat c}
\liealg{\ml}{k} \liealg{\mlp}{k'}
\newcommand{\ML}{K}
\newcommand{\Pow}{\mathrmsl{P}}
\newcommand{\typ}{\cI}
\newcommand{\sm}{\psi}
\newcommand{\cm}{\varphi}
\newcommand{\grph}{{\mathrmsl\Gam}}
\newcommand{\cgph}{{\mathrmsl\Delta}}
\newcommand{\edge}[1]{\lower1pt\hbox{$\stackrel{\lower1pt\hbox
{$\scriptstyle \smash{#1}$}}{\text{---}}$}}
\newcommand{\op}{\mathrm{op}}
\newcommand{\bk}[1][\g]{\mathscr{W}^{#1}}
\newcommand{\bc}[1][\g]{\mathscr{B}^{#1}}
\newcommand{\Ac}[1][\g]{\mathscr{A}^{#1}}
\newcommand{\PF}[1][\g]{\mathscr{P}^{#1}}
\newcommand{\PI}[1]{\grph^{#1}}
\newcommand{\PP}[1]{\Pi_{#1}}
\def\BK/{BK}
\def\grsys/{system}
\newcommand{\cstd}[1]{_{#1\text{\rm-co}}}
\newcommand{\wop}[1]{_{#1\text{\rm-op}}}
\begin{document}
\title[Parabolic subalgebras and parabolic projection]
{Parabolic subalgebras,\\ parabolic buildings and parabolic projection}
\author{David M. J. Calderbank}
\email{D.M.J.Calderbank@bath.ac.uk}
\address{Mathematical Sciences\\ University of Bath\\
Bath BA2 7AY\\ UK.}
\author{Passawan Noppakaew}
\email{passawan@su.ac.th}
\address{Department of Mathematics\\ Faculty of Science\\ Silpakorn University\\
Nakhon Pathom 73000\\ Thailand.}
\date{\thismonth}
\begin{abstract} Reductive (or semisimple) algebraic groups, Lie groups and
Lie algebras have a rich geometry determined by their parabolic subgroups and
subalgebras, which carry the structure of a building in the sense of J. Tits.
We present herein an elementary approach to the geometry of parabolic
subalgebras, over an arbitrary field of characteristic zero, which does not
rely upon the structure theory of semisimple Lie algebras. Indeed we derive
such structure theory, from root systems to the Bruhat decomposition, from the
properties of parabolic subalgebras.  As well as constructing the Tits
building of a reductive Lie algebra, we establish a ``parabolic projection''
process which sends parabolic subalgebras of a reductive Lie algebra to
parabolic subalgebras of a Levi subquotient.  We indicate how these ideas may
be used to study geometric configurations and their moduli.
\end{abstract}
\maketitle

Parabolic subgroups and their Lie algebras are fundamental in Lie theory and
the theory of algebraic
groups~\cite{Bou:gal,Hel:dglg,Hum:lar,Kna:lgbi,Mil:lag,OnVi:lgla,Pro:lg,Tits:casg,Wal:rrg}.
They play a key role in combinatorial, differential and integrable geometry
through Tits buildings and parabolic invariant
theory~\cite{AbBr:bta,BuCo:dg,BDPP:is,CpSl:pg,Sch:hig,Spr:it,Tits:bst}.
Traditional approaches define parabolic subalgebras of semisimple or reductive
Lie algebras as those containing a Borel (maximal solvable) subalgebra in some
field extension, and typically develop the theory of parabolic subalgebras
using the root system associated to a Cartan subalgebra of such a Borel
subalgebra. Such approaches are far from elementary, and provide limited
insight when the field is not algebraically closed.

The present paper is motivated by a programme to study geometric
configurations in projective spaces and other generalized flag manifolds (\ie,
adjoint orbits of parabolic subalgebras) using a process called
\emph{parabolic projection}. This process relies upon two observations: first,
if $\p,\q$ are parabolic subalgebras of a (reductive or semisimple) Lie
algebra $\g$, then so is $\p\cap\q+\nil(\q)$ (where $\nil(\q)$ is the
nilpotent radical of $\q$); secondly, if $\r$ is a Lie subalgebra of $\q$, and
$\q$ is parabolic subalgebra of $\g$, then $\r$ is a parabolic subalgebra of
$\g$ if and only if it is the inverse image of a parabolic subalgebra of the
reductive Levi quotient of $\q_0:=\q/\nil(\q)$---thus two senses in which $\r$
might be called a parabolic subalgebra of $\q$ coincide. These observations
point towards a theory of parabolic subalgebras of arbitrary Lie algebras,
which may be developed using straightforward methods of linear algebra over an
arbitrary field of characteristic zero, without relying on the structure
theory of semisimple Lie algebras. Indeed, much of the latter can be derived
as a consequence.

This development combines two approaches to parabolic subalgebras. In the
first approach, inspired by work of V. Morozov~\cite{Mor:one,Mor:ptr,PaVi:Mor}
and other work emphasising the role of nilpotent elements in Lie algebra
theory (see also~\cite{Bou:gal,Mil:lag,Pom:rft,Spr:it}), the Borel subalgebras
in the traditional definition of parabolic subalgebras are replaced by the
normalizers of maximal nil subalgebras (\ie, subalgebras contained in the
nilpotent cone).  In the second approach, initiated by
A. Grothendieck~\cite{Gro:p1} and developed by F. Burstall~\cite{Bur:rps} and
others (see~\cite{BDPP:is,BuRa:ttrss,CDS:rcd,Cla:phd}), a subalgebra of a
semisimple Lie algebra is parabolic if its Killing perp is a nilpotent
subalgebra. The latter definition appears to rely upon the Killing form of a
semisimple Lie algebra, but is easily adapted to the reductive case.

The set of parabolic subalgebras of a reductive Lie algebra $\g$ has a rich
structure: it may be decomposed into flag varieties under the action of the
adjoint group $G$ of $\g$, but is also equipped with an incidence geometry
relating these varieties. Let us illustrate this in the case that $G=\PGL(V)$
is the projective general linear group of a vector space $V$ of dimension
$n+1$. Among the parabolic subalgebras of $\g$, the maximal (proper) parabolic
subalgebras play a distinguished role. The corresponding flag varieties are
the grassmannians $\Gr_k(V)$ ($1\leq k\leq n$) of $k$-dimensional subspaces of
$V$. Two such subspaces are incident if one contains the other, and more
general parabolic subalgebras (and their flag varieties) can be described
using sets of mutually incident subspaces (called ``flags'').  In particular,
the minimal parabolic subalgebras of $\g$ are the infinitesimal stabilizers of
complete flags, which are nested chains of subspaces of $V$, one of each
dimension.

The general picture is similar. Among adjoint orbits of parabolic subalgebras
(generalized flag varieties), the maximal ones are distinguished. For
algebraically closed fields, these may be identified with the nodes of the
Dynkin diagram, and with the nodes of the Coxeter diagram of the (restricted)
root system in general. Other parabolic subalgebras may be described by the
maximal (proper) parabolic subalgebras containing them (which are mutually
incident), so that their adjoint orbits correspond to subsets of nodes of the
Coxeter or Dynkin diagram. This situation is abstracted by J. Tits' theory of
buildings~\cite{Tits:bst}, which deserves (in our opinion) a more central
place in Lie theory and representation theory than it currently enjoys.

Our purpose in this paper is to present a self-contained treatment of
parabolic subalgebras and their associated Lie theory, sufficient to give a
novel proof that the set of minimal parabolic subalgebras of a Lie algebra are
the chambers of a strongly transitive Tits building. In the algebraically
closed case, this establishes the conjugacy theorems for Cartan subalgebras
and Borel subalgebras, together with the Bruhat decomposition, using no
algebraic geometry and very little structure theory. It also provides a
natural framework in which to develop the basic properties of parabolic
projection.  In subsequent work, we shall apply parabolic projection to the
construction of geometric configurations and discrete integrable geometries.

Although the results herein are essentially algebraic in nature, we wish to
emphasise the geometry behind them. We thus begin in Section~\ref{s:eig} by
introducing incidence systems, and two examples which we shall use throughout
the paper to illustrate the theory. Example~\ref{ex:1A} is the incidence
geometry of vector (or projective) subspaces of a vector (or projective) space
over an arbitrary field, while Example~\ref{ex:1B} is the incidence geometry
of isotropic subspaces of a real inner product space of indefinite signature.

Section~\ref{s:elt} provides a self-contained treatment of the Lie algebra
theory we need, modulo some basic facts about Jordan decompositions and
invariant bilinear forms which we summarize in Appendix~\ref{s:a}. In order to
work over an arbitrary field of characteristic zero, we develop the
nilpotent--reductive dichotomy for Lie algebras rather than the
solvable--semisimple dichotomy. We thus ignore Lie's theorem, Weyl's theorem,
the Levi--Malcev decomposition, and even most of the representation theory of
$\sgl_2$. Instead, we emphasise the role played by filtrations and Engel's
theorem, nilpotency ideals and the nilpotent cone, trace forms and Cartan's
criterion (which, in its most primitive form, concerns nilpotency rather than
solvability).  In Theorem~\ref{t:red} we thus establish, in a novel way, the
basic result that (in characteristic zero) a Lie algebra is reductive if and
only if it admits a nondegenerate trace form.  Using this,
Proposition~\ref{p:rcc} extends one of Cartan's criteria from $\gl(V)$ to
reductive Lie algebras, a result which we have not been able to find in the
literature.

We define parabolic subalgebras in Section~\ref{s:ps} as those subalgebras
containing the normalizer of a maximal nil subalgebra, but immediately obtain,
in Theorem~\ref{t:par-equiv} several equivalent definitions using
``admissible'' trace forms and nilpotency ideals, some of which are new,
cf.~\cite{Bou:gal,Bur:rps}. We then consider pairs of parabolic subalgebras,
their ``oppositeness'', and their incidence properties.  Intersections of
opposite minimal parabolic subalgebras define minimal Levi subalgebras, also
known as anisotropic kernels, which govern the structure theory of root
systems for semisimple Lie algebras in characteristic zero---see
Theorem~\ref{t:std-pars}. The theory of such ``restricted'' root systems is
well known in the context of algebraic groups~\cite{Tits:casg} or when the
underlying field is the real
numbers~\cite{Hel:dglg,Kna:lgbi,OnVi:lgla,Wal:rrg}; here, though, we are
forced to discard concepts such as Cartan decompositions which are particular
to the real case.

The core results of the paper appear in Section~\ref{s:gta}, where we show
that the parabolic subalgebras of a reductive Lie algebra $\g$ form the
simplices of a Tits' building. This theory has a formidable reputation, but
has become more approachable in recent years as the key concepts have become
better understood. More recent approaches emphasise chamber
\grsys/s~\cite{Eve:sib,Ron:lb,Tits:lab,Wei:ssb} rather than simplicial
complexes~\cite{Gar:bcg,Tits:bst}.  In particular, these approaches make
explicit the labelling corresponding to the nodes of the Coxeter--Dynkin
diagram in the parabolic case. Unfortunately, the modern definition of
buildings incorporates an abstraction of the Bruhat decomposition, which is a
nontrivial result in representation theory. In order to address these issues,
we adopt a concise hybrid approach to buildings, which combines the original
viewpoint (using ``apartments'') with more recent approaches using chamber
\grsys/s. Following~\cite{AbBr:bta,Ron:lb,Wei:ssb}, we develop sufficient
theory to derive an abstract Bruhat decomposition
(Theorem~\ref{t:abstract-bruhat}) for strongly transitive buildings. We use
this to obtain (in Theorem~\ref{t:bruhat}) the Bruhat decomposition from
conjugacy results for minimal parabolic subalgebras and their Levi
subalgebras, which we also prove. This result is well-known in the context of
algebraic groups, using Tits' \emph{$(B,N)$-pairs}, but our methods are almost
entirely different.

In Section~\ref{s:ppgc}, we turn finally to \emph{parabolic projection},
which, for a given fixed parabolic subalgebra $\q\leq\g$, projects arbitrary
parabolic subalgebras of $\g$ onto parabolic subalgebras of the reductive Levi
quotient $\q_0$ of $\q$. This idea was originally developed by A. Macpherson
and the first author~\cite{Mac:scsp} using root systems and standard parabolic
subgroups.  The analysis here is based instead on Proposition~\ref{p:par-pair}
which gives an explicit formula for the projection.  The main result is
Theorem~\ref{t:pp-morphism}, in which parabolic projection is shown to be a
morphism of chamber \grsys/s when restricted to weakly opposite
subalgebras. This is the basis for constructions of geometric configurations
(see~\cite{Nop:pp}) that we shall pursue elsewhere.

\acknowledge

The first author thanks Vladimir Sou\v cek, Paul Gauduchon, Robert Marsh, Tony
Dooley, Daniel Clarke and Amine Chakhchoukh for helpful discussions, and the
Eduard \v Cech Institute, grant number GA CR P201/12/G028, for financial
support.  We are also immensely grateful to Fran Burstall and Alastair
King for their comments and ideas.


\section{Elements of incidence geometry}\label{s:eig}

Incidence geometry is conveniently described using graph theory, where graphs
(herein) are undirected with no loops or multiple edges. We use $|\grph|$ and
$E_\grph$ to denote the vertex and edge sets of a graph $\grph$; an edge is
determined by its two endpoints, so $E_\grph$ may be viewed as a collection of
two element subsets of $|\grph|$, or equivalently, a symmetric irreflexive
relation on $|\grph|$.  For $v,w\in |\grph|$, we write $v\edge{} w$ for the
reflexive closure of this relation (which is the structure preserved by graph
morphisms).  By a ``subgraph'', we always mean a subset of vertices with the
induced relation.

\begin{defn}[\cite{Sch:hig,Shu:pl,Wei:ssb}] An \emph{incidence system} over
a (usually finite) set $\typ$ is an $\typ$-multipartite graph $\grph$, \ie, a
graph equipped with a type function $t=t_\grph\colon|\grph|\to\typ$ such that
$\forall\,v,w\in|\grph|$, $v\edge{}w$ and $t(v)=t(w)$ imply $v=w$. An
\emph{incidence morphism} $\grph_1\to\grph_2$ of incidence systems over $\typ$
is a type-preserving graph morphism.

A \emph{flag} or \emph{clique} in $\grph$ of type $J\in\Pow(\typ)$ (\ie,
$J\sub\typ$) is a set of mutually incident elements, one of each type $j\in J$
(a \emph{$J$-flag}). We denote the set of $J$-flags by $\cF\grph(J)$; together
with the obvious ``face maps'' $\cF\grph(J_2)\to\cF\grph(J_1)$ for $J_1\sub
J_2$, these form an (abstract, $\typ$-labelled) simplicial complex $\cF\grph$
(a presheaf or functor $\Pow(\typ)^{op}\to\mathbf{Set}$) called the \emph{flag
  complex}---$\cF\grph$ is also an incidence system over $\Pow(\typ)$: two
flags are incident iff their union is a flag.  A \emph{full flag} is an
$\typ$-flag $\sigma\in\cF\grph(\typ)$, \ie, $\sigma$ contains one element of
each type $j\in\typ$.
\end{defn}

\begin{example}\label{ex:1A} The proper nonempty subsets $B$ of an $n+1$
element set $\cS$ form an incidence system $\grph^\cS$ over $\typ_n:=
\{1,2,\ldots n\}$, where $t(B)$ is the number of elements of $B$, and
$B_1\edge{}B_2$ iff $B_1\sub B_2$ or $B_2\sub B_1$. We ``linearize'' this
example as follows.

Let $V$ be a vector space of dimension $n+1$ over a field $\F$.  The proper
nontrivial subspaces $W\leq V$ are the elements of an incidence system
$\grph^V$ over $\typ_n$, where $t(W)=\dim W$ and $W_1\edge{}W_2$ iff $W_1\leq
W_2$ or $W_2\leq W_1$. For $J\sub\typ_n$, a $J$-flag is a family of subspaces
$W_j:j\in J$ of $V$ with $\dim W_j=j$ and $W_j\leq W_k$ for $j\leq k$. Thus a
full flag is a nested sequence $0\leq W_1\leq W_2\leq\cdots \leq W_n\leq V$
with $\dim W_j=j$.
\end{example}
\begin{example}\label{ex:1B} Let $U$ be a vector space of dimension $2n+k$
over $\R$ equipped with a quadratic form $Q_U$ of signature $(n+k,n)$, where
$k\geq 1$. The nontrivial isotropic subspaces of $U$ (on which $Q_U$ is
identically zero) have $1\leq\dim U\leq n$, and are the elements (typed by
dimension) of an incidence system $\grph^{U,Q_U}$ over $\typ_n$, where the
incidence relation is again given by containment. We shall provide a discrete
model for $\grph^{U,Q_U}$ in Example~\ref{ex:4B}.
\end{example}

\section{Elementary Lie theory}\label{s:elt}

\subsection{Lie algebra notions and notations}

Recall that a \emph{Lie algebra} $\g$ over a field $\F$ is an $\F$-vector
space equipped with a skew-symmetric bilinear operation
$[\cdot,\cdot]\colon\g\times\g\to\g$ satisfying the Jacobi identity
$[x,[y,z]]=[[x,y],z]+[y,[x,z]]$. The commutator bracket $(\alpha,\beta)\mapsto
[\alpha,\beta]=\alpha\circ\beta-\beta\circ\alpha$ makes $\End_\F(V)$ into a
Lie algebra, denoted $\gl(V)$. A \emph{representation} of Lie algebra $\g$ on
a vector space $V$ is a Lie algebra homomorphism $\rho\colon\g\to\gl(V)$ (\ie,
a linear map with $\rho([x,y])=[\rho(x),\rho(y)]$). We write $\rho(x,v)$ or
$x\act v$ as a shorthand for the action $\rho(x)(v)$ of $x\in\g$ on $v\in V$.
For subspaces $\p\sub\g$ and $U\sub V$, we define $\rho(\p,U)=\p\act U$ to be
the span of $\{\rho(x)(u)\st x\in \p,\; u\in U\}$, and introduce shorthands
$\rho(x,U)=x\act U$ and $\rho(\p,u)=\p\act u$ when $\p=\spn\{x\}$ or
$U=\spn\{u\}$.

The action of subspaces $\p\sub\g$ on subspaces $U\sub V$ has an upper adjoint
in each variable:
\begin{align*}
\p\act U\sub W \quad&\text{iff}\quad
\p\sub \c_{\g}(U,W):=\{x\in\g\st x\act U\sub W\}\\
&\text{iff}\quad U\sub \c_{V}(\p,W):=\{v\in V\st \p\act v\sub W\}.
\end{align*}
In particular $\c_{\g}(U,U)=\stab_\g(U)$ is the \emph{stabilizer} of $U$ and
$\c_{V}(\p,0)=\ker\rho(\p)$ is the (\emph{joint}) \emph{kernel} of the $\p$
action; $U$ is \emph{$\p$-invariant} in $V$ iff $\p\act U\sub U$ iff
$\p\sub\stab_\g(U)$.

The \emph{adjoint representation} $\ad\colon \g\to \gl(\g)$ is defined by
$\ad(x)(y)=[x,y]$ for $x,y\in \g$ (which is a representation by the Jacobi
identity). Thus $\ad(\p,\q)$ is the bracket $[\p,\q]$ of subspaces $\p,\q\sub
\g$ and we set $[x,\q]:=[\spn\{x\},\q]$. The upper adjoint specializes to
give:
\begin{equation*}
[\p,\q]\sub \r\quad\text{iff}\quad
\p\sub \c_{\g}(\q,\r):=\{x\in\g\st [x,\q]\sub \r\}.
\end{equation*}
In particular $\n_{\g}(\q):=\c_{\g}(\q,\q)$ is the \emph{normalizer} of $\q$,
$\c_{\g}(\q):=\c_{\g}(\q,0)$ is the \emph{centralizer} of $\q$, and
$\z(\g):=\c_{\g}(\g)$ is the \emph{centre} of $\g$.  Thus $\p\sub \g$ is a
\emph{subalgebra} ($\p\leq\g$) iff $[\p,\p]\sub\p$ iff $\p\sub\n_{\g}(\p)$,
and an \emph{ideal} ($\p\idealin\g$) iff $[\g,\p]\sub\p$ iff $\n_{\g}(\p)=\g$.
We note a useful lemma.

\begin{lemma}\label{l:dsum-cent} If $\b\sub\a\sub\h$ and $\g=\h\dsum\m$,
where $\h\leq\g$ and $[\h,\m]\sub\m$, then $\c_{\g}(\a,\b)=\c_{\h}(\a,\b)\dsum
\bigl(\c_{\g}(\a)\cap\m\bigr)$.
\end{lemma}

A Lie algebra $\g$ is \emph{reductive} if it has a faithful semisimple
representation (see Appendix~\ref{a:sr}).  This holds in particular if $\g$ is
nonabelian with irreducible adjoint representation (\ie, $\g$ has no proper
nontrivial ideals); then $\g$ is said to be \emph{simple}.  More generally,
the adjoint representation of $\g$ is faithful and semisimple if and only if
$\g$ is \emph{semisimple}, \ie, a direct sum of simple ideals. Thus any
semisimple Lie algebra $\g$ is reductive with $[\g,\g]=\g$.

\subsection{Filtered and graded Lie algebras}

\begin{defn} A \emph{$\Z$-graded vector space} is a vector space $V$ equipped
with a \emph{$\Z$-grading}, \ie, a direct sum decomposition $V=\Dsum_{k\in \Z}
V_k$. A \emph{filtration} of a vector space $V$ is a family $V^{(k)}:k\in\Z$
of subspaces of $V$ such that $i\leq j\; \Rightarrow\; V^{(i)}\sub V^{(j)}$.

Let $V^+:=\bigcup_{k\in\Z} V^{(k)}$ and $V_-:=\bigcap_{k\in\Z} V^{(k)}$. If
$V^+=V$ and $V_-=0$, we say $V$ is a \emph{filtered vector space} and refer to
$\gr(V):=\Dsum_{k\in \Z} V^{(k)}/V^{(k-1)}$ as the \emph{associated graded
  vector space}.
\end{defn}

\begin{defn}[See \textit{e.g.}~\cite{CpSl:pg}] A \emph{$\Z$-graded Lie algebra}
is a Lie algebra $\g$ equipped with a $\Z$-grading $\g=\Dsum_{k\in \Z} \g_k$
such that $\forall\,i,j\in\Z$, $[\g_i,\g_j]\sub\g_{i+j}$.

A \emph{filtration} of a Lie algebra $\g$ is a filtration $\f^{(k)}:k\in\Z$ of
the underlying vector space such that $\forall\,i,j\in \Z$,
$[\f^{(i)},\f^{(j)}] \sub\f^{(i+j)}$.  If $\g_k:=\f^{(k)}/\f^{(k-1)}$, this
induces a Lie algebra structure on $\gr_\f(\g):=\Dsum_{k\in\Z}\g_k$. If
$\f^+=\g$ and $\f_-=0$, we say $\g$ is a \emph{filtered Lie algebra} with
\emph{associated graded Lie algebra} $\gr_\f(\g)$.
\end{defn}

\begin{prop}\label{p:pair} Let $\f^{(-1)}:=\n\leq\g$, $\f^{(0)}:=\p\leq\n_{\g}(\n)$,
and for $j>0$ inductively define $\f^{(-j-1)}=[\n,\f^{(-j)}]$ and
$\f^{(j)}=\c_{\g}(\n,\f^{(j-1)})$.  Then for all $i,j\in\Z$,
$[\f^{(i)},\f^{(j)}]\sub\f^{(i+j)}$ and if $\f^{(-1)}\sub\f^{(0)}$ \textup(\ie,
$\n\idealin\p$\textup) then $\f^{(k)}:k\in\Z$ is a filtration of $\g$.
\end{prop}
\begin{proof} By construction $[\f^{(-1)},\f^{(j)}]\sub \f^{(j-1)}$ for
all $j\in\Z$, and $\f^{(j-1)}\sub \f^{(j)}$ for $j\neq 0$. Next, by Jacobi, for
any $i>0$, $[\f^{(-i-1)},\f^{(j)}]= [[\f^{(-1)},\f^{(-i)}],\f^{(j)}]\sub
[\f^{(-1)},[\f^{(-i)},\f^{(j)}]] + [\f^{(-i)},\f^{(j-1)}]$ for all $j\in\Z$,
so induction on $i$ shows that $[\f^{(-i)},\f^{(j)}]\sub \f^{(-i+j)}$ for
$i>0$ and $j\in\Z$.

We now show $[\f^{(i)},\f^{(j)}]\sub\f^{(i+j)}$ for $i,j\geq 0$ by induction
on $i+j$: $[\f^{(0)},\f^{(0)}]\sub \f^{(0)}$ since $\f^{(0)}=\p$ is a
subalgebra, and for $i+j>0$, Jacobi implies $[[\f^{(i)},\f^{(j)}],\n]\sub
[\f^{(i)},\f^{(j-1)}]+[\f^{(i-1)},\f^{(j)}]\sub\f^{(i+j-1)}$ (inductively),
\ie, $[\f^{(i)},\f^{(j)}]\sub\c_{\g}(\n,\f^{(i+j-1)})=\f^{(i+j)}$.
\end{proof}
We call this the filtration of $\g$ \emph{induced by $\n\idealin\p$}, or
\emph{by $\n$} if $\p=\n_{\g}(\n)$. Its negative part is the \emph{lower
  central series} of $\n$, and if $\f^{(-k)}=0$ for sufficiently large $k$, we
say $\n$ is \emph{nilpotent}.

\begin{defn} Let $\g,\f^{(k)}:k\in\Z$ be a filtered Lie algebra.  A
\emph{filtration} of a representation $\rho\colon\g\to\gl(V)$ of $\g$ is a
filtration $V^{(k)}:k\in\Z$ of $V$ such that $\forall\,i,j\in\Z$,
$\f^{(i)}\act V^{(j)}:=\rho(\f^{(i)}, V^{(j)})\sub V^{(i+j)}$. This induces a
representation $\overline\rho$ of $\gr_\f(\g)$ on $\gr(V):=\Dsum_{k\in\Z}V_k$,
where $V_k=V^{(k)}/V^{(k-1)}$, such that $\forall\,i,j\in\Z$,
$\overline\rho(\g_i,V_j)\sub V_{i+j}$. If $V$ is filtered, we call $V$ (or
$\rho$) a \emph{filtered representation} of $\g$ with \emph{associated graded
  representation} $\gr(V)$ (or $\overline\rho$).
\end{defn}

\subsection{Engel's theorem and nilpotence}

\begin{thm}[Engel] Let $\rho\colon\g\to \gl(V)$ be a finite dimensional
representation of a Lie algebra $\g$ and $\n\idealin\g$. Then the following
are equivalent\textup:
\begin{numlist}
\item $\rho(x)$ is nilpotent for all $x\in \n$\textup;
\item $\n$ acts trivially on any irreducible subquotient of $\rho$\textup;
\item $V$ is a filtered representation for the filtration of $\g$ induced by
$\n$.
\end{numlist}
\end{thm}
\begin{proof} The key is the following well-known lemma due to Engel;
we omit the proof.
\begin{lemma}\label{l:eng} Let $\u\leq\gl(W)$ be a Lie subalgebra with
$\sigma$ nilpotent for all $\sigma\in\u$\textup; then if $W$ is nonzero,
  $\exists w\in W$ nonzero such that for all $\sigma\in \u$, $\sigma(w)=0$.
\end{lemma}
(1)$\Rightarrow$(2). Let $\rho'\colon \g\to\gl(W)$ be an irreducible
subquotient of $\rho$ and let $W'=\{w\in W\st \rho'(\n)(w)=0\}$. Since
$\n\idealin\g$, $W'$ is $\g$-invariant. Since $\n$ acts by nilpotent
endomorphisms, Lemma~\ref{l:eng} implies $W'$ is nonzero; hence $W'=W$.

(2)$\Rightarrow$(3). This is an easy induction on $\dimn V$: if $V=0$, we are
done; otherwise $V$ has a nontrivial irreducible $\g$-invariant subspace
$U\leq V$, which is in the kernel of $\rho(\n)$ by (2). By induction, $V/U$ by
is filtered by some $V^{(j)}/U:j\in\Z$. If $k$ denotes the largest integer
with $V^{(k)}=U$, we may redefine $V^{(j)}=0$ for $j<k$ to make $V$ into a
filtered representation.

(3)$\Rightarrow$(1). Given such a filtration $V^{(j)}:j\in\Z$, we may assume
$V^{(0)}=V$ and $V^{(j)}=0$ for $j<-k$. Now any $x\in\n$ satisfies
$\rho(x)^{k+1}=0$.
\end{proof}

If any of these conditions hold, $\n$ is called a \emph{nilpotency
  ideal}~\cite{Bou:gal} for $\rho$; it follows from (3) that $\rho(\n)$ is a
nilpotent Lie algebra. By (2), $\g$ has a largest nilpotency ideal
$\nil_\rho(\g)$ for $\rho$, namely the intersection of the kernels of the
simple subquotients of $\rho$. Thus $\kernel\rho\leq\nil_\rho(\g)$ and equality
holds if $\rho$ is semisimple.  Hence $\nil_\rho(\g)=0$ if $\rho$ is faithful
and semisimple.

Henceforth, we assume $\g$ is finite dimensional. An ideal $\n\idealin\g$ is a
nilpotency ideal for the adjoint representation $\ad$ of $\g$ if and only if
it is nilpotent, and so $\nil_{\ad}(\g)$ is the largest nilpotent ideal of
$\g$, often called the \emph{nilradical}.

\begin{defn} The \emph{nilpotent radical} $\nil(\g)\idealin\g$ is the
intersection of its largest nilpotency ideals, or equivalently, the
intersection of the kernels of the simple representations of $\g$.
\end{defn}

Since $\nil(\g)\leq\nil_{\ad}(\g)$, it is a nilpotent ideal, and it is the
intersection of the kernels of finitely many simple representations of $\g$,
hence the kernel of a semisimple representation of $\g$. Thus $\g/\nil(\g)$ is
reductive, and $\g$ is reductive if and only if $\nil(\g)=0$. The Lie algebra
$\g$ is a filtered Lie algebra via the \emph{canonical filtration}
$\g^{(k)}:k\in\Z$ induced by $\nil(\g)\idealin\g$, and with respect to the
canonical filtration, any (finite dimensional) representation of $\g$ is a
filtered representation by Engel's theorem.

\begin{defn} The \emph{nilpotent cone} of a Lie algebra $\g$ is $\cN(\g)=
\{x\in \g\st \rho(x)$ is nilpotent for any representation $\rho$ of $\g\}$.  A
\emph{nil subalgebra} $\n\leq\g$ is a Lie subalgebra with $\n\sub \cN(\g)$.
\end{defn}
\begin{rems} If $f\colon\h\to\g$ is a Lie algebra homomorphism, then any
representation $\rho$ of $\g$ induces a representation $\rho\circ f$ of $h$,
so $f(\cN(\h))\sub\cN(\g)$. Clearly $x\in \cN(\g)$ if and only if $x$ is
nilpotent in any \emph{semisimple} representation of $\g$.  Thus $\cN(\g)$
is the inverse image of the nilpotent cone in $\g/\nil(\g)$, which is reductive.
A nil subalgebra $\n\leq\g$ is a nilpotency ideal for the adjoint
representation of $\n_{\g}(\n)$ on $\g$, hence nilpotent.
\end{rems}

\begin{example}\label{ex:2A} Let $V$ be a vector space with $\dim_\F V=n+1$ as
in Example~\ref{ex:1A}.  Then $\gl(V)$ is a Lie algebra whose nilpotent cone
$\cN(\gl(V))$ consists of the nilpotent endomorphisms of $V$.  A subalgebra
$\n$ of $\gl(V)$ is therefore a nil subalgebra iff there is a filtration of
$V$ such that $\n$ acts trivially on the associated graded representation.
With respect to a basis adapted to such a filtration, the elements of $\n$ are
strictly upper triangular.  In particular, $\gl(V)$ is reductive, and its
derived algebra is the subalgebra $\sgl(V)= [\gl(V),\gl(V)]$ of traceless
endomorphisms. Thus a nil subalgebra of $\gl(V)$ is a nilpotent subalgebra of
$\sgl(V)$.
\end{example}
\begin{example}\label{ex:2B} Let $U,Q_U$ be as in Example~\ref{ex:1B}, and
let $B_U$ be the associated symmetric bilinear form of signature $(n+k,n)$ on
$U$.  Then $\so(U,Q_U)=\{A\in \gl(U)\st B_U(A u_1,u_2)+B_U(u_1,A u_2)\}$ is a
Lie subalgebra of $\gl(U)$, and $\cN(\so(U,Q_U))$ again consists of the
elements of $\so(U,Q_U)$ which are nilpotent endomorphisms of $U$. If $\n$ is
a nonzero nil subalgebra of $\so(U,Q_U)$ then $\n\act U$ is nontrivial, hence
so is its intersection with $\c_U(\n,0)=\bigcap_{A\in\n}\ker A$ (because $\n$
acts trivially on any irreducible summand of $\n\act U$). The intersection
$(\n\act U)\cap\c_U(\n,0)$ is isotropic, hence contains a $1$-dimensional
isotropic subspace $\pi_1$. Applying the same argument inductively to
$\pi_1^\perp/\pi_1$, we obtain a filtration $0\leq\pi_1\leq\pi_2\leq\cdots
\leq\pi_n\leq\pi_n^\perp\leq\cdots\leq \pi_2^\perp\leq\pi_1^\perp\leq U$, with
$\dim\pi_j=j$, preserved by $\n$, and $\n$ acts trivially on the associated
graded representation.
\end{example}

\subsection{Invariant forms and trace-forms} We summarize basic properties
of invariant symmetric bilinear forms in Appendix~\ref{a:if}.

\begin{defn} An invariant (symmetric bilinear) form on a filtered Lie
algebra $\g$ is \emph{compatible} with the filtration iff $\f^{(j-1)}\sub
(\f^{(-j)})^\perp$ for all $j\in \Z$. The restriction of a compatible
invariant form to $\f^{(j)}\times \f^{(-j)}$ descends to a pairing $\g_j\times
\g_{-j}\to \F$ and hence induces a invariant form on $\gr_\f(\g)$, called the
\emph{associated graded invariant form}.
\end{defn}

\begin{prop} If $\f^{(j-1)} = (\f^{(-j)})^\perp$ for all $j\in \Z$, the
associated graded invariant form on $\gr_\f(\g)$ is nondegenerate, \ie,
${\g_j}^\perp=\Dsum_{k\neq -j} \g_k$.
\end{prop}
\begin{proof} $x$ is in $\gr_\f(\g)^\perp$ if and only if its homogeneous
components are. Now for $x\in \g_j$, we have $x\in\gr_\f(\g)^\perp$ if and
only if $x$ is orthogonal to $\g_{-j}$, \ie, any lift $\tilde x$ to $\f^{(j)}$
is in $(\f^{(-j)})^\perp$. On the other hand $x=0$ if and only if the lift is
in $\f^{(j-1)}$.
\end{proof}

\begin{defn} The \emph{trace form} on $\g$ associated to a representation $
\rho\colon\g\to\gl(V)$, is the invariant form
$(x,y)\mapsto\trace(\rho(x)\rho(y))$.
\end{defn}

\begin{prop} Let $\rho\colon\g\to \gl(V)$ be a filtered representation of a
filtered Lie algebra. Then the trace form of $\rho$ is compatible with the
filtration $\f^{(j)}:j\in\Z$ of $\g$, and the associated graded invariant form
on $\gr_\f(\g)$ is the trace form of the induced representation
$\overline\rho$.
\end{prop}
\begin{proof} If $x\in \f^{(j-1)}$ and $y\in \f^{(-j)}$ then $\rho(x)\rho(y)$
maps $V^{(k)}$ to $V^{(k-1)}$ for all $k\in \Z$, hence is nilpotent. Thus
$\rho(x)\rho(y)$ is trace-free, and hence $x$ and $y$ are orthogonal.

To compute the associated graded form on $x\in \g_j$ and $y\in \g_{-j}$,
choose lifts $\tilde x\in\f^{(j)}$, $\tilde y\in \f^{(-j)}$ and a splitting
$V\cong \gr(V)$ of the filtration of $V$. Then the trace of $\rho(\tilde x)
\rho(\tilde y)$ may be computed by restricting and projecting
onto $V_k$, for each $k\in\Z$, computing the trace, and summing over $k$,
which yields the trace of $\overline\rho(x)\overline\rho(y)$.
\end{proof}

\begin{prop} Let $\ip{\cdot,\cdot}$ be the trace form of
$\rho\colon\g\to\gl(V)$.
\begin{numlist}
\item $\nil_\rho(\g)\leq \g^{\perp}$ and the induced invariant form on
$\g/\nil_\rho(\g)$ is a trace form.
\item If $\ip{\cdot,\cdot}$ is nondegenerate, then $\nil_\rho(\g)=0$.
\item For a Lie homomorphism $f\colon\h\to \g$, $\ip{\cdot,\cdot}$ pulls back to
a trace form associated to $\rho\circ f$.
\item If $\h\leq\g$ is a subalgebra and $\rho(\h)$ is nilpotent, then $\h\leq
  \n_{\g}(\h)^\perp$.
\end{numlist}
\end{prop}
\begin{proof} (1) By Engel's theorem, $V$ is a filtered
representation for the filtration of $\g$ induced by $\nil_\rho(\g)$; hence
the trace-form is compatible and $\nil_\rho(\g)\leq \g^\perp$.

(2) is immediate from (1) and (3) is obvious.

(4) Pull back (\ie, restrict) $\ip{\cdot,\cdot}$ to $\n_{\g}(\h)$; since $\h$
is a nilpotency ideal for the restriction of $\rho$ to $\n_{\g}(\h)$,
$\h\sub\nil_\rho(\n_{\g}(\h)) \sub\n_{\g}(\h)^\perp\cap
\n_{\g}(\h)\sub\n_{\g}(\h)^\perp$.
\end{proof}

\begin{cor}\label{c:tf} Let $\ip{\cdot,\cdot}$ be a trace form on $\g$.
\begin{numlist}
\item $\nil(\g)\leq \g^{\perp}$ and the induced invariant form on
$\g/\nil(\g)$ is a trace form.
\item If $\ip{\cdot,\cdot}$ is nondegenerate, then $\g$ is reductive.
\item For a Lie homomorphism $f\colon\h\to \g$, $\ip{\cdot,\cdot}$ pulls back
to a trace form on $\h$.
\item If $\h\leq\g$ is a nil subalgebra, then $\h\leq\n_{\g}(\h)^\perp$.
\end{numlist}
\end{cor}

\subsection{Cartan criteria for nilpotency and reductive Lie algebras}

In this subsection, we assume the underlying field $\F$ has characteristic
zero (hence is perfect).  Then we have the following straightforward
characterization of the nilpotent cone $\cN(\g)$.

\begin{prop}\label{p:nil-cone-char} Let $\g$ be a Lie algebra. Then $x\in
\cN(\g)$ if and only if $x\in [\g,\g]$ and $\ad(x)$ is nilpotent.  In
particular $\nil(\g)=\nil_{\ad}(\g)\cap[\g,\g]$.
\end{prop}
\begin{proof} If $x\in \cN(\g)$ then $\ad x$ is nilpotent and $x$ is trivial
in any $1$-dimensional representation. Conversely, if $x\in [\g,\g]$ with
$\ad(x)$ nilpotent, then $\rho(x)$ is nilpotent in any semisimple
representation $\rho\colon\g\to \gl(V)$ by Proposition~\ref{p:J2}.  Hence
$x\in \cN(\g)$.
\end{proof}

\begin{lemma}\label{l:cc} Suppose $\gl(V)=\h\dsum\m$ with $\h\leq\g$, $\m
\sub\h^\perp$ and $[\h,\m]\sub \m$. For subspaces $\b\sub \a\sub \h$, let
$\u=\c_{\h}(\a,\b)$.  Then any element of $\u\cap\u^\perp\sub\gl(V)$ is
nilpotent.
\end{lemma}
\begin{proof}[\pfin{Bou:gal,Hum:lar}] By Lemma~\ref{l:dsum-cent}, $\c_{\gl(V)}(
\a,\b)=\u\dsum(\c_{\gl(V)}(\a)\cap\m)$, and $\m\sub\h^\perp$ so
$\u^\perp\cap\h\sub \c_{\gl(V)}(\a,\b)^\perp$. Hence it suffices to prove the
result for $\m=0$.

Let $x=x_s+x_n$ be the Jordan decomposition of $x\in\u\leq\gl(V)$, let $\F^c$
be a splitting field for $x_s$, let $\cS\sub\F^c$ be the set of eigenvalues
${x_s}^c$, and let $f\colon \F^c\to \Q$ be a $\Q$-linear form on
$\F^c$. Define $y\in\gl(V^c)$ to be scalar multiplication by $f(\lam)$ on the
$\lam$-eigenspace of $x_s$ for all $\lam\in\cS$. Then by Lemma~\ref{l:ad-ss},
$\ad y$ is a polynomial with no constant term in $\ad {x_s}^c$, hence in $\ad
x^c$, so that $y\in\u^c$.  If also $x\in\u^\perp$, then $0=\trace(x^c
y)=\sum_{\lam\in\cS} m(\lam) \lam f(\lam)$, where $m(\lam)\in\Z^+$ is the
multiplicity of $\lam$. Applying $f$, we obtain $f(\lam)=0$ for all
$\Q$-linear forms $f$ and all $\lam\in\cS$. Hence $x_s=0$ and $x$ is
nilpotent.
\end{proof}

\begin{prop}[Cartan's criterion]\label{p:cc} Suppose $\rho\colon\g\to\gl(V)$
is a representation.  Then, with respect to the induced trace form,
$\g^\perp\cap[\g,\g]\leq\nil_\rho(\g)$.
\end{prop}
\begin{proof} Take $\b=\rho(\g^\perp)$, $\a=\rho(\g)$ and $\m=0$
in Lemma~\ref{l:cc}, so that $\u=\c_{\gl(V)}(\rho(\g),\rho(\g^\perp))$ (and
hence $\rho(\g^\perp)\idealin\u$).  Since
$\u\leq\c_{\gl(V)}(\rho(\g),\rho(\g)^\perp)$,
$\rho([\g,\g])=[\rho(\g),\rho(\g)]\leq \u^\perp$.  Hence
$\rho(\g^\perp\cap[\g,\g])\leq \u\cap\u^\perp$, and so
$\g^\perp\cap[\g,\g]\idealin\g$ is a nilpotency ideal for $\rho$.
\end{proof}
\begin{thm}\label{t:red} A finite dimensional Lie algebra $\g$ over a field of
characteristic zero is reductive if and only if it admits a nondegenerate
trace form.  Then $\g=\z(\g)\dsum[\g,\g]$, the sum is orthogonal, and
$[\g,\g]\cong \ad(\g)=\c_{\der(\g)}(\z(\g))\cong\der([\g,\g])$ is semisimple.
\end{thm}
\begin{proof} If $\g$ is reductive, it has a faithful semisimple
representation $\rho$, and $\g^\perp\cap[\g,\g]=0$ with respect to the induced
trace form by Proposition~\ref{p:cc}, so $\g^\perp\sub[\g,\g]^\perp=
\c_{\g}(\g,\g^\perp)=\z(\g)$ by Proposition~\ref{a:pif} (1).  Now $[\g,\g]^\perp=
\z(\g)\sub\nil_{\ad}(\g)$ has trivial intersection with $[\g,\g]$ by
Proposition~\ref{p:nil-cone-char}, so the trace form is nondegenerate on
$[\g,\g]$.  Since $\g^\perp\leq\z(\g)$, we may add a representation of
$\g/[\g,\g]$ to $\rho$ to make the trace form nondegenerate. Conversely, the
existence of a nondegenerate trace form implies $\g$ is reductive by
Corollary~\ref{c:tf}.

Since $\nil_{\ad}(\g)\cap[\g,\g]=0$, Dieudonn\'e's famous
observation~\cite{Die:slg} (see Proposition~\ref{a:rad-ss}) shows that
$[\g,\g]$ is semisimple.  Thus $[\g,\g]\cong\ad(\g)$ is a nondegenerate ideal
in $\der(\g)\leq\gl(\g)$, and hence $\h:=\ad(\g)^\perp\cap\der(\g)$ is a
complementary ideal in $\der(\g)$.  Now for any $D\in \h$ and $x\in\g$,
$0=[D,\ad(x)]=\ad(D(x))$, so $D(x)\in\z(\g)$ and hence $D$ vanishes on
$[\g,\g]$; thus $\ad(\g)=\c_{\der(\g)}(\z(\g))\cong\der([\g,\g])$.
\end{proof}

\begin{defn} An \emph{admissible form} on a Lie algebra $\g$ is a trace form
with $\g^\perp=\nil(\g)$.
\end{defn}
Admissible forms always exist: just pull back a nondegenerate trace form on
the reductive quotient $\g/\nil(\g)$. They can be used to construct linear
subspaces of $\cN(\g)$.

\begin{prop}\label{p:rcc} Let $\b\sub\a\sub\g$ be subspaces of a Lie
algebra $\g$ which contain $\nil(\g)$ and let $\u=\c_{\g}(\a,\b)$. Then for
any admissible form on $\g$, $\u\cap\u^\perp\sub \cN(\g)$.
\end{prop}
\begin{proof} Any admissible form is induced by a semisimple
representation $\rho\colon \g\to\gl(V)$. Now $\gl(V)=\rho(\g)\dsum
\rho(\g)^\perp$ and $\rho(\u)= \c_{\rho(\g)}(\rho(\a),\rho(\b))$, so for any
$x\in \u\cap\u^\perp$, $\rho(x)$ is nilpotent by Lemma~\ref{l:cc}, and so
$\ad(\rho(x))$ is nilpotent on $\rho(\g)$. Since $\c_{\g}(\g,\g^\perp)\sub
\u$, $\u^\perp\sub [\g,\g]$ by Proposition~\ref{a:pif} (1) and so
$\rho(x)\sub[\rho(\g),\rho(\g)]$.  Thus $\rho(x)\in \cN(\rho(\g))$ by
Proposition~\ref{p:nil-cone-char}, and so $x\in\cN(\g)$.
\end{proof}

\section{Parabolic subalgebras}\label{s:ps}

\subsection{General definition} Henceforth, the characteristic of underlying  
field $\F$ will be zero.

\begin{defn} A subalgebra $\p$ of a Lie algebra $\g$ is
\emph{parabolic in $\g$} if it contains the normalizer $\n_{\g}(\m)$ of a
maximal nil subalgebra $\m\sub\cN(\g)$ of $\g$.
\end{defn}
Any Lie algebra $\g$ is a parabolic subalgebra of itself, and more generally,
if $\p\leq\q\leq\g$ are subalgebras such that $\p$ is parabolic in $\g$ then
$\q$ is parabolic in $\g$. At the other extreme, the minimal parabolic
subalgebras of $\g$ are the normalizers of its maximal nil subalgebras.

\begin{prop} Let $\p$ be a subalgebra of $\g$ such that $\nil(\p)$
contains $\nil(\g)$. Then $\p$ is parabolic in $\g$ if and only if
$\p/\nil(\g)$ is parabolic in $\g/\nil(\g)$.
\end{prop}
\begin{proof} Since $\nil(\g)$ is an ideal in $\g$ and $\cN(\g/\nil(\g))=
\cN(\g)/\nil(\g)$, $\m$ is a maximal nil subalgebra of $\g$ if and only if
$\m$ contains $\nil(\g)$ and $\m/\nil(\g)$ is a maximal nil subalgebra of
$\g/\nil(\g)$; now $\n_{\g}(\m)/\nil(\g)=\n_{\g/\nil(\g)}(\m/\nil(\g))$ and
the result follows.
\end{proof}
For $\q\leq\g$, define $\nil_\g(\q):=\nil_{\ad_\g}(\q)\cap[\g,\g]$, the
largest ideal of $\q$ contained in $\cN(\g)$.

\begin{prop}\label{p:max-nil} Suppose $\nil(\g)\leq\m\leq\g$ and
let $\q=\n_{\g}(\m)$.  Then for any admissible form on $\g$, \textup{(1)}
$\m=\nil_{\g}(\q)\implies \m=\q\cap\q^\perp$, and \textup{(2)}
$\m=\q\cap\q^\perp\implies\m=\q^\perp$.
\end{prop}
\begin{proof}[\pfin{Bou:gal,Mor:ptr,PaVi:Mor}] Since $\g^\perp\leq\m
\idealin\q\leq \g$, $[\q,\q^\perp]\leq \q^\perp$ and $\q\cap\q^\perp$ is an
ideal in $\q$, which is contained in $\cN(\g)$ by Proposition~\ref{p:rcc}.  If
$\m\sub\cN(\g)$ then $\m\sub\q\cap\q^\perp$ by Corollary~\ref{c:tf} (4), and
(1) follows. Now $\q$ acts on $\q^\perp/\m$ with $\m$ as a nilpotency
ideal. However, the $\m$-action has kernel $(\q\cap\q^\perp)/\m=0$. Hence
$\q^\perp=\m$ by Engel's theorem.
\end{proof}

This result leads to the following equivalences,
cf.~\cite{Bou:gal,Bur:rps,CDS:rcd,Gro:p1}.

\begin{thm}\label{t:par-equiv} For a subalgebra $\p\leq\g$ with $\nil(\g)\leq
\nil(\p)$, the following are equivalent.
\begin{numlist}
\item There is an admissible form on $\g$ such that $\p^\perp$ is a
nilpotent subalgebra of $[\g,\g]$.
\item There is an admissible form on $\g$ such that $\p^\perp$ is a nil
subalgebra of $\g$.
\item $\p$ is parabolic in $\g$.
\item For any admissible form on $\g$, $\p^\perp\sub\p$ and $\n_{\g}(\p)=\p$.
\item $\p=\n_{\g}(\nil_{\g}(\p))$.
\item For any admissible form on $\g$, $\p^\perp=\nil_{\g}(\p)=\nil(\p)$.
\item $\dimn \g-\dimn\p =\dimn\nil(\p)-\dimn\nil(\g)$.
\end{numlist}
\end{thm}
\begin{proof} By Corollary~\ref{c:tf}, $\nil(\p)\leq\nil_{\g}(\p)\leq\p^\perp$
for any admissible form on $\g$; also, since $\g^\perp\leq\p\leq\g$,
$\p\sub\n_{\g}(\p)=\n_{\g}(\p^\perp)=[\p,\p^\perp]^\perp$ by
Proposition~\ref{a:pif} (1). In particular (6) and (7) are equivalent, since
$\dimn \p^\perp=\dimn\g-\dimn\p+\dimn\g^\perp$ and $\g^\perp=\nil(\g)$.

(1)$\Leftrightarrow$(2) (cf.~\cite{Bur:rps}). Since $\p^\perp\leq\g$ and $\p\leq
\n_{\g}(\p^\perp)$, Proposition~\ref{p:pair} yields subspaces $\p^{(j)}$ of
$\g$ with $\p^{(-1)}=\p^\perp$, $\p^{(0)}=\p$ and $[\p^\perp,\p^{(j)}]\leq
\p^{(j-1)}$ for all $j\in\Z$. Now $\c_{\g}(\p^\perp,(\p^{(-j)})^\perp)=
[\p^\perp,\p^{(-j)}]^\perp$, so induction on $j$ shows that
$\p^{(j-1)}=(\p^{(-j)})^\perp$ for $j>0$, and hence $\bigcup_{k\geq 0}\p^{(k)}
=\bigl(\bigcap_{k<0}\p^{(k)}\bigr)^\perp=\{0\}^{\perp}=\g$, since $\p^\perp$
is nilpotent.  Hence $\p^\perp$ is an $\ad_\g$-nilpotent subalgebra of
$[\g,\g]$, \ie, a nil subalgebra of $\g$ by Proposition~\ref{p:nil-cone-char}.
The converse is immediate.

(2)$\Leftrightarrow$(3). $\p^\perp$ is a nil subalgebra of $\g$ iff
$\p^\perp\sub\m$ iff $\m^\perp\sub\p$ for a maximal nil subalgebra $\m$, and
by Proposition~\ref{p:max-nil}, $\m^\perp=\n_{\g}(\m)$.

(1--3)$\Rightarrow$(4). Corollary~\ref{c:tf} (4) implies $\p^\perp\sub
\n_{\g}(\p^\perp)^\perp$, hence $\p^\perp\sub\n_{\g}(\p^\perp)\sub \p^{\perp\perp}=\p$.

(4)$\Rightarrow$(5). Since $\p^\perp\sub\p=\n_{\g}(\p^\perp)$,
$\p^\perp=\p\cap\p^\perp\idealin\p$ and is nil by Proposition~\ref{p:rcc}, so
$\p^\perp\leq\nil_\g(\p)$; hence equality holds and $\p=\n_{\g}(\nil_\g(\p))$.

(5)$\Rightarrow$(6) By Proposition~\ref{p:max-nil}, $\nil_{\g}(\p)=\p^\perp$,
and hence $\p=\n_{\g}(\p^\perp)=[\p,\p^\perp]^\perp$, so $\p^\perp\sub
[\p,\p^\perp]+\g^\perp\leq\nil(\p)$ since $[\p,\nil_\g(\p)]\sub\nil(\p)$.

(6)$\Rightarrow$(1). Immediate, since $\nil(\p)\sub[\p,\p]$.
\end{proof}

\begin{prop} Suppose $\p\leq\q\leq\g$ are subalgebras with $\q$ parabolic
in $\g$. Then $\p$ is parabolic in $\q$ if and only if it is parabolic in
$\g$.
\end{prop}
\begin{proof} Fix an admissible form on $\g$; its restriction to $\q$ is
admissible, since $\q^\perp=\nil(\q)$. Since $\p\sub\q\sub\g$, we have
$\g^\perp\sub\q^\perp\sub \p^\perp$, \ie, $\nil(\g)\sub \nil(\q)\sub
\p^\perp\cap\q$.

If $\p$ is parabolic in $\q$, $\q^\perp=\nil(\q)\sub \p$ and
$\nil(\p)=\p^\perp\cap\q$.  It follows that $\nil(\g)\sub\nil(\p)$ and
$\p^\perp\sub\q^{\perp\perp}=\q$, hence $\nil(\p)=\p^\perp$, as required.
Conversely if $\p$ is parabolic in $\g$, then $\p^\perp=\nil(\p)$ contains
$\nil(\q)$, and since $\p^\perp\sub \p\sub \q$, $\p$ is parabolic in $\q$.
\end{proof}

\begin{cor}\label{c:levi-par} Let $\p\leq\q\leq\g$ with $\g$ reductive,
and $\q$ parabolic in $\g$. Then $\p$ is parabolic in $\g$ if and only if
$\nil(\p)$ contains $\nil(\q)$ and $\p/\nil(\q)$ is parabolic in
$\q/\nil(\q)$.
\end{cor}

We refer to $\q_0=\q/\nil(\q)$ as the (reductive) \emph{Levi quotient} of
$\q$.

\subsection{Grading elements and splittings}

\begin{defn} A \emph{grading element} for a Lie algebra $\p$ is an element
$\chi\in \gr(\p)$ (the associated graded algebra of the canonical filtration)
with $[\chi,x]=jx$ for all $j\in \Z$ and $x\in \p_j$. A (reductive) \emph{Levi
  subalgebra} is a Lie subalgebra of $\p$ complementary to $\nil(\p)$.
\end{defn}
A grading element $\chi$ exists if and only if the derivation $D$ of $\gr(\p)$
defined by $Dx=jx$ (for all $j\in\Z$ and $x\in \p_j$) is inner; $\chi$ then
belongs to the centre of $\p_0$, and is unique modulo the centre of $\gr(\p)$.
If $\p$ is reductive, \ie, $\nil(\p)=0$, then $0\in\p=\p_0$ is a grading
element.

\begin{prop}[see \cite{BDPP:is,CDS:rcd}]\label{p:h-split} If $\p$ has a
grading element $\chi\in\p_0$, then the following sets are in canonical
$\exp(\nil(\p))$-equivariant bijection\textup:
\begin{bulletlist}
\item lifts of $\chi$ to $\p$\textup;
\item splittings $\p\cong\gr(\p)$ of the canonical filtration of $\p$\textup;
\item Levi subalgebras of $\p$.
\end{bulletlist}
Furthermore, the action of $\exp(\nil(\p))$ is free and transitive.
\end{prop}
\begin{proof} A lift of $\chi$ to $\p$ determines a splitting $\gr(\p)\cong \p$
via its eigenspace decomposition. For any such splitting, the image of $\p_0$
is a Levi subalgebra of $\p$, and any Levi subalgebra of $\p$ contains a
unique lift of $\chi$. Since $\p_0=\p/\nil(\p)$, lifts of $\chi$ form an
affine space modelled on $\nil(\p)$, which is nilpotent, so $\exp(\nil(\p))$
acts freely and transitively on the lifts.
\end{proof}

For a parabolic subalgebra $\p\leq\g$, we let $\p^{(j)}$ be the filtration of
$\g$ induced by $\nil(\p)\idealin\p\leq\g$ (cf. Theorem~\ref{t:par-equiv}),
and denote the associated graded Lie algebra by $\gr_\p(\g)$.

\begin{prop} Let $\p$ be parabolic in a reductive Lie algebra $\g$.  Then the
associated graded algebra $\gr_\p(\g)$ is reductive, and its centre is
contained in $\p_0=\p/\nil(\p)$.
\end{prop}
\begin{proof} Since $\g$ is reductive, it admits a nondegenerate trace form,
induced by a representation $\rho\colon\g\to \gl(V)$.  Since
$\p^\perp=\nil(\p)$, it acts nilpotently on $V$, and hence induces a
filtration on $V$ with $V^{(j)}=V$ for $j\geq 0$ (say) and
$V^{(j-1)}=\rho(\p^\perp,V^{(j)})$ for $j\leq 0$.  This makes $V$ into a
filtered representation of $\g$ (with the filtration by $\p^{(j)}:j\in\Z$).

We have already observed that $\p^{(j-1)}=(\p^{(-j)})^\perp$ for $j>0$, hence
writing $k=1-j$ and taking perps, we have $(\p^{(-k)})^\perp=\p^{(k-1)}$ for
$k\leq 0$. The induced trace form on $\gr_\p(\g)$ is therefore nondegenerate,
and hence $\gr_\p(\g)$ is reductive.

For any $x\in\z(\gr_\p(\g))$, its homogeneous components are also in the
centre.  Now for a homogeneous central element
$x\in\p_j:=\p^{(j)}/\p^{(j-1)}$, any lift centralizes $\nil(\p)$, so it must
lie in $\p$, \ie, $j\leq 0$. However, since $\p^{(k)}=[\nil(\p),\p^{(k+1)}]$
for $k<0$, we must also have $j\geq 0$, \ie, the centre of $\gr_\p(\g)$ is
contained in $\p_0$.
\end{proof}
\begin{cor} A parabolic subalgebra $\p$ of a reductive Lie algebra $\g$ has
a unique grading element $\chi\in\z(\p_0)\cap[\gr_\p(\g),\gr_\p(\g)]$.
\end{cor}
Indeed, the derivation $D$ of $\gr_\p(\g)$ defined by $Dx=jx$ for $x\in\g_j$
vanishes on the centre of $\gr_\p(\g)$ and preserves its semisimple
complement.  It follows that $D$ is an inner derivation, \ie, $D=\ad\chi$ for
$\chi\in \z(\p_0)$ which is determined uniquely by requiring it is in the
complement $[\gr_\p(\g),\gr_\p(\g)]$ to centre of $\gr_\p(\g)$.

\subsection{Pairs of parabolic subalgebras}

\begin{defn}[see \cite{Cla:phd}] Parabolic subalgebras $\p,\q$ of a reductive
Lie algebra $\g$ are said to be \emph{costandard} if $\p\cap\q$ is parabolic
in $\g$.
\end{defn}
\begin{prop}\label{p:costd-equiv} For parabolics $\p,\q\leq\g$, with $\g$
reductive, the following are equivalent\textup:
\begin{numlist}
\item $\p$ and $\q$ are costandard\textup;
\item $\p\cap\q$ contains a minimal parabolic subalgebra of $\g$\textup;
\item for some \textup(hence any\textup) admissible form on $\g$,
$\p^\perp\leq \q$ \textup(or equivalently, $\q^\perp\leq \p$\textup).
\end{numlist}
\end{prop}
\begin{proof} The first two conditions are manifestly equivalent. Now for any
admissible form on $\g$, $(\p\cap\q)^\perp=\p^\perp+\q^\perp$. Thus (1) and (2)
imply (3). For the converse, it suffices to show that (3) implies
$\p^\perp+\q^\perp$ is nilpotent. For this, let $\r^{(-i)}=\sum_{j\geq0,k\geq0,i=j+k}
\p^{(-i)}\cap \q^{(-j)}$. Then $\r^{(-1)}=\p^\perp+\q^\perp$, and
$[\r^{(-1)},\r^{(-i)}]\sub \r^{(-i-1)}$, so $\r^{(-1)}$ is nilpotent.
\end{proof}

\begin{prop}\label{p:costd-pars} If $\p^1,\ldots \p^k\leq\g$ are
parabolic and pairwise costandard in a reductive Lie algebra $\g$, then
$\p^1\cap\cdots\cap \p^k$ is parabolic in $\g$.
\end{prop}
\begin{proof} Induction on $k$, with $k=1$ being trivial:
if $\p^1\cap\cdots\cap \p^{k-1}$ is parabolic, and $\p^k$ is costandard with
each $\p^i$, then $(\p^k)^\perp\sub\p^i$ for all $i=1,\ldots k-1$, hence
$(\p^k)^\perp\sub\p^1\cap\cdots\cap\p^{k-1}$, which proves the result by
Proposition~\ref{p:costd-equiv}.
\end{proof}

\begin{defn} Parabolic subalgebras $\p,\hp$ of a reductive Lie algebra $\g$
are said to be \emph{opposite} if (for some, hence any, admissible form)
$\p+\hp^\perp=\g$ (\ie, $\p^\perp\cap\hp=0$) and $\p^\perp+\hp=\g$ (\ie,
$\p\cap\hp^\perp=0$); in other words, $\g=\p\dsum\hp^\perp$ (\ie, $\p\cap\hp$
is a Levi subalgebra of $\hp$) which means equivalently that
$\g=\p^\perp\dsum\hp$ (\ie, $\p\cap\hp$ is a Levi subalgebra of $\p$).
\end{defn}
\begin{rems}\label{r:weyl-tors} If $\p\leq\g$ is a parabolic subalgebra, then
the set of parabolic subalgebras $\hp$ opposite to $\p$ is an
$\exp(\nil(\p))$-torsor (a \emph{$G$-torsor} for a group $G$ is a simply
transitive $G$-set); Proposition~\ref{p:h-split} yields canonical isomorphisms
between the following $\exp(\nil(\p))$-torsors:
\begin{bulletlist}
\item lifts of the grading element $\chi\in\p_0$ to $\p=\p^{(0)}$\textup;
\item splittings $\gr_\p(\g)\cong \g$\textup;
\item parabolic subalgebras $\hp$ opposite to $\p$ in $\g$\textup;
\item Levi subalgebras of $\p$.
\end{bulletlist}
An (algebraic) \emph{Weyl structure} for $\p$ is an element of this (\ie, any
of these) torsor(s). An opposite pair $\p,\hp\leq\g$ gives a vector space
direct sum decomposition of $\g$ into subalgebras $\nil(\p)\dsum
\p\cap\hp\dsum \nil(\hp)$, where $\p\cap\hp$ is a Levi subalgebra of both $\p$
and $\hp$, hence is orthogonal to both $\nil(\p)$ and $\nil(\hp)$ with respect
to any admissible form on $\g$. Thus any admissible form on $\g$ restricts to
an admissible form on $\p\cap\hp$ and a duality between $\nil(\p)$ and
$\nil(\hp)$.
\end{rems}
\begin{prop}\label{p:par-pair} Let $\p$ and $\q$ be parabolic in $\g$.
Then $\r:=\p\cap\q+\nil(\q)$ is parabolic in $\g$ with
$\nil(\r)=\nil(\p)\cap\q+\nil(\q)$.
\end{prop}
\begin{proof} It suffices to prove that $\r/\nil(\q)$ is parabolic
in $\q_0:=\q/\nil(\q)$.  Introduce an admissible form on $\g$ and the induced
admissible form on $\q_0$.  First note that
$(\p\cap\q+\q^\perp)^\perp=(\p\cap\q)^\perp\cap\q =(\p^\perp+\q^\perp)\cap
\q$.  Since $\q^\perp\leq\q$, this gives $\r^\perp = \p^\perp\cap\q +
\q^\perp$.  Then $(\r/\nil(\q))^\perp= \r^\perp/\nil(\q)$ is contained in
$\r/\nil(\q)$ and in $[\q_0,\q_0]$, since
$\p^\perp=[\p,\p^\perp]+\nil(\g)$. Now $\r^\perp/\nil(\q)$ is nilpotent since
$\p^\perp=\nil(\p)$ is nilpotent.
\end{proof}

\begin{prop}\label{p:par-pair-tors} Let $\p,\q$ be parabolic subalgebras of a
reductive Lie algebra $\g$. Then the following are canonically isomorphic
$\exp(\nil(\p)\cap\nil(\q))$-torsors\textup:
\begin{bulletlist}
\item lifts $\xi_\p$ and $\xi_\q$ of the grading elements of $\p$ and $\q$
  with $[\xi_\p,\xi_\q]=0$ \textup(thus
  $\xi_\p,\xi_\q\in\p\cap\q$\textup)\textup;
\item Levi subalgebras of $\p$ in $\p\cap\q$\textup;
\item Levi subalgebras of $\q$ in $\p\cap\q$.
\end{bulletlist}
\end{prop}
\begin{proof} Let $\r=\p\cap\q+\nil(\q)$. Then $\r$ is a parabolic subalgebra
of $\g$ contained in $\q$, so that $\r/\nil(\q)$ is a parabolic subalgebra of
$\q_0:=\q/\nil(\q)$. The grading element $\chi_\q$ of $\q$ lies in the centre of
$\q/\nil(\q)$ and hence in $\r/\nil(\q)$. Consequently the lifts of $\chi_\q$
belong to an affine subspace of $\r$ modelled on $\nil(\q)$, which therefore
meets $\p\cap\q$ in an affine subspace modelled on $\p\cap\nil(\q)$. Mutatis
mutandis, there is an affine subspace of lifts of the grading element of $\p$
to $\p\cap\q$, modelled on $\nil(\p)\cap\nil(\q)$.

Now if $\xi_\p$ and $\xi_\q$ are lifts of grading elements of $\p$ and $\q$ to
$\p\cap\q$, then $[\xi_\p,\xi_\q]$ belongs to $\nil(\p)\cap \nil(\q)$, on
which $\ad\xi_\p$ and $\ad\xi_\q$ are invertible. Hence the equation
$[\xi_\p,\xi_\q]=0$ uniquely determines either lift to $\p\cap\q$ from the
other, and the compatible lifts form an affine space modelled on
$\nil(\p)\cap\nil(\q)$, hence a torsor for $\exp(\nil(\p)\cap\nil(\q))$.
\end{proof}

\subsection{Minimal Levi subalgebras and maximal anisotropic subalgebras}

We refer to Levi subalgebras of parabolic subalgebras of $\g$ as Levi
subalgebras of $\g$. In particular a \emph{minimal Levi subalgebra} of $\g$ is
a Levi subalgebra of a minimal parabolic subalgebra $\pb\leq\g$.

\begin{cor}\label{c:par-reg-apt-comp} If $\p,\q\leq\g$ are parabolic,
$\p\cap\q$ contains a minimal Levi subalgebra of $\g$.
\end{cor}

Indeed, $\p$ contains a minimal parabolic subalgebras $\pb$ and by
Proposition~\ref{p:par-pair-tors} there is a minimal Levi subalgebra in
$\pb\cap\q\leq\p\cap\q$.

\begin{prop} If $\h$ is a Levi subalgebra of $\g$ then $\cN(\h)=\cN(\g)\cap\h$,
and if $\h$ is a minimal Levi subalgebra, every $x\in\h$ is semisimple
in \textup(the adjoint representation of\textup) $\g$.
\end{prop}
\begin{proof} Any admissible form on $\g$ induces an admissible form on $\h$
(see Remarks~\ref{r:weyl-tors}). Now by Corollary~\ref{c:tf}, any $x\in
\cN(\g)\cap\h$ is also in $\c_{\g}(x)^\perp\cap\h\sub\z(\h)^\perp\cap\h=[\h,\h]$.
Since $x$ is $\ad_\h$-nilpotent, it is in $\cN(\h)$, and the other inclusion
is automatic. If $\h$ is minimal, $\cN(\h)=0$ (else $\h$ would contain a proper
parabolic subalgebra). Now $\h=\c_{\g}(\z(\h))$ and hence
$\ad_\g(\h)=\c_{\der(\g)}(\z(\h))$, which is closed under Jordan decomposition
by Proposition~\ref{p:Jordans}.
\end{proof}

\begin{defn} A subalgebra $\k$ of a reductive Lie algebra $\g$ is called
\emph{anisotropic}, \emph{toral} or \emph{ad-semisimple} if every
element of $\k$ is semisimple in (the adjoint representation of) $\g$.
\end{defn}
Over an algebraically closed field, anisotropic subalgebras are
abelian~\cite{Hum:lar,Mil:lag}.
\begin{prop}\label{p:ml} Let $\k$ be an anisotropic subalgebra of $\g$
containing a lift $\xi_\p$ of the grading element of a parabolic $\p$. Then
$\xi_\p\in\z(\k)$, \ie, $\k$ is in the Levi subalgebra $\c_{\g}(\xi_\p)$.
\end{prop}
\begin{proof} Since $\ad(\xi_\p)$ is semisimple (with integer eigenvalues),
the invariant subspace $\k \leq\g$ is a direct sum of eigenspaces for
$\ad(\xi_\p)\restr{\k}$. Now if $[\xi_\p,x]=jx$ for some $x\in \k$ and $j\in
\Z$, then $\ad(x)(\xi_\p)=-jx$ and $\ad(x)^2(\xi_\p)=0$. Now $\ad(x)$ is
semisimple, so $x=0$ or $j=0$.
\end{proof}
\begin{cor} Minimal Levi subalgebras $\h\leq\g$ are maximal anisotropic
subalgebras.
\end{cor}
The converse does not necessarily hold unless the underlying field is
algebraically closed, in which case minimal Levi subalgebras are abelian and
called \emph{Cartan subalgebras}. A minimal parabolic subalgebra $\b$ with
abelian Levi factor is solvable (\ie, $[\b,\b]$ is nilpotent) and called a
\emph{Borel subalgebra}.

Let $\ml$ be a minimal Levi subalgebra of $\g$, also called an
\emph{anisotropic kernel}. Since $\z(\ml)$ is anisotropic and abelian, there
is a splitting field for its action on $\g$, \ie, a field extension $\F^c$
such that the adjoint action of $\z(\ml)$ on $\g^c=\g\tens_\F\F^c$ is
simultaneously diagonalizable. Let $\a$ be the subspace of $\z(\ml)$ whose
elements have all eigenvalues in $\F$.

\begin{defn} For $\alpha\in\a^*$, let $\g_\alpha=\{x\in\g\st [h,x]=\alpha(h) x
\text{ for all } h\in\a\}$. If $\alpha\neq 0$ and $\g_\alpha\neq 0$, we say
that $\alpha\in\a^*$ is a (restricted) \emph{root} of $\g$, and call
$\g_\alpha$ the \emph{root space} of $\alpha$.  The \emph{root lattice}
$\Lam_r$ of $(\g,\ml)$ is the \textup(free\textup) $\Z$-submodule of $\a^*$
generated by the set $\Phi$ of roots. Elements of its dual
$\Lam_{cw}=\{\xi\in\a\st \alpha(\xi)\in\Z$ for all $\alpha\in\Phi\}$ are
called \emph{coweights}.
\end{defn}
Since $\ml$ is the centralizer of $\z(\ml)$, $\g$ has a root space
decomposition
\begin{equation*}
\g=\ml\dsum \Dsum_{\alpha\in\Phi} \g_\alpha,
\end{equation*}
with $\g^*_\alpha\cong\g_{-\alpha}$ (using any admissible form on $\g$). The
kernels of the roots have intersection $\z(\g)$ in $\a=\z(\g)\dsum(\a\cap
[\g,\g])$, hence span its annihilator $\z(\g)^\circ\cong(\a\cap[\g,\g])^*$.
We shall also need the following basic fact concerning the existence of
\emph{coroots} $h_\alpha:\alpha\in\Phi$ in $\a$.

\begin{prop}\label{p:coroots} For any $\alpha\in\Phi$, there is a
unique $h_\alpha\in\a\cap[\g_\alpha,\g_{-\alpha}]$ with $\alpha(h_\alpha)=2$;
furthermore, for any nonzero $x_\alpha\in\g_\alpha$, there exists
$y_\alpha\in\g_{-\alpha}$ with $h_\alpha=[x_\alpha,y_\alpha]$.
\end{prop}
\begin{proof}[\pfin{Bou:gal,Hum:lar}] For any $x\in\g_\alpha$ and $y\in
\g_{-\alpha}$, if $[x,y]\in\kernel\alpha\sub\a$ then $x,y,[x,y]$ span a
subalgebra $\s\leq\g$ with $[x,y]\in[\s,\s]\cap\z(\s)\sub\cN(\s)$, so
$[x,y]\in[\g,\g]\cap\ml$ is nilpotent and semisimple, hence zero. For
$x=x_\alpha\neq 0$, $[x_\alpha,\g_{-\alpha}]\cap\a=\c_{\g}(x_\alpha)^\perp\cap\a=
(\ker\alpha)^\perp\cap\a$ so there is a unique $h_\alpha\in\a$ with
$\alpha(h_\alpha)=2$ and $[x_\alpha,y_\alpha]=h_\alpha$ for some
$y_\alpha\in\g_{-\alpha}$.
\end{proof}
For any parabolic subalgebra $\p$ containing $\ml$, the unique lift $\xi_\p$
of its grading element to $\z(\ml)$ acts on $\g$ with integer eigenvalues, so
$\xi_\p\in\Lam_{cw}$. For any $\xi\in\Lam_{cw}$ and $j\in\Z$, let
$\Phi^j_\xi=\{\alpha\in\Phi\st \alpha(\xi)=k\}$ and
$\Phi^\pm_{\xi}=\{\alpha\in\Phi\st \pm\alpha(\xi)\in\Z^+\}$. Then:
\begin{bulletlist}
\item $\Phi^\pm_\xi$ and $\Phi^0_\xi$ are (relatively) additively closed in
  $\Phi$, with $\Phi=\Phi^-_\xi\sqcup\Phi^0_\xi\sqcup\Phi^+_\xi$;
\item $\p_\xi:=\ml\dsum\Dsum_{\alpha\in\Phi\setminus\Phi^+_\xi} \g_\alpha$ is
  a parabolic subalgebra with nilpotent radical
  $\p^{\,\perp}_\xi=\Dsum_{\alpha\in\Phi^-_\xi}\g_\alpha$;
\item for any parabolic $\p\supseteq\ml$, $\{\xi\in\Lam_{cw} \st \p_\xi=\p\}$
  is additively closed and contains $\xi_\p$.
\end{bulletlist}

\begin{rem} It may be illuminating to compare the above theory with standard
approaches to the theory of real reductive Lie
algebras~\cite{Hel:dglg,Kna:lgbi,OnVi:lgla,Wal:rrg}, in which the main
ingredient is a Cartan decomposition of $\g$, \ie, a symmetric decomposition
$\g=\h\dsum\m$ (into the $+1$ and $-1$ eigenspaces of an involution) such that
$\h$ is a maximal compact subalgebra. Then $\a$ is a maximal abelian subspace
of $\m$ and $\k=(\k\cap\h)\dsum\a$ is the centralizer of $\a$ in $\g$. It
follows that $\a$ is the ``split part'' $\t\cap\m$ of a ``maximally split''
Cartan subalgebra $\t$, where $\t\cap\h$ is a Cartan subalgebra of
$\k\cap\h$. The (restricted) roots in $\a^*$ are restrictions of roots in
$\t^*$. We shall refer to $\a\leq\ml$ in general as a \emph{split Cartan
  subalgebra} of $\g$.
\end{rem}
\begin{example}\label{ex:3A} Let $V$ be as in
Examples~\ref{ex:1A}--\ref{ex:2A}. The parabolic subalgebras of $\gl(V)$ are
the stabilizers of flags in the incidence system $\grph^V$ of proper
nontrivial subspaces $W\leq V$. Elements $W_1,W_2$ are incident if and only if
their stabilizers (maximal proper parabolics) are costandard. A minimal
parabolic subalgebra is the stabilizer of full flag, \ie, a Borel subalgebra,
while a minimal Levi subalgebra $\ml$ is the stabilizer of a direct sum
decomposition of $V$ into one dimensional subspaces, \ie, a Cartan subalgebra
of $\gl(V)$. If $\F=\R$ and $V=\R^n$, then $\so_n\dsum\R\,\iden$ is a maximal
anisotropic subalgebra of $\gl_n(\R)$, but not a minimal Levi subalgebra.
\end{example}
\begin{example}\label{ex:3B}  Let $U,Q_U$ be as in
Examples~\ref{ex:1B} and \ref{ex:2B}. The parabolic subalgebras of
$\so(U,Q_U)$ are stabilizers of flags in the incidence system $\grph^{U,Q_U}$,
with incident flags corresponding to costandard parabolics.  A minimal
parabolic subalgebra is the stabilizer of a full flag
$\pi_1\leq\pi_2\leq\cdots\leq\pi_n$ of isotropic subspaces, while a minimal
Levi subalgebra $\ml$ is the stabilizer of an orthogonal decomposition $U=
\bigl(\Dsum \cR_\ml)\dsum W_\ml$ where $Q_U$ is positive definite on $W_\ml$,
$\dim W_\ml=k$, and the $n$ elements of $\cR_\ml$ have signature $(1,1)$.  If
$W_\ml\cong\R^k$ and $\Dsum \cR_\ml\cong\R^{n,n}$, then
$\ml\cong\so_k\dsum\a\sub\so_{n+k,n}$ where $\a$ is a Cartan subalgebra of
$\so_{n,n}$. Thus $\so_{n+k}\dsum\so_n\sub\so_{n+k,n}$ is maximal anisotropic,
but not a minimal Levi subalgebra.
\end{example}

\subsection{Minimal parabolic subalgebras and lowest weight representations}

Let $\pb$ be a minimal parabolic subalgebra containing a minimal Levi
subalgebra $\ml$.  The corresponding Weyl structure $\xi_{\pb}$ induces a
grading $\g=\Dsum_{j\in\Z} \pb_j$ such that $\pb$ is the nonpositive part,
with $\pb_0=\ml$ and $\pb_j=\Dsum_{\alpha\in\Phi^j}\g_\alpha$ for $j\neq 0$
(where $\Phi^j=\Phi^j_{\xi_\pb}$). Thus $\Phi^0=\empt$ and $\Phi$ is the
disjoint union of \emph{positive roots} $\alpha\in\Phi^+$ and negative roots
$\alpha\in\Phi^-$. For $\lam,\mu\in\a^*$, we write $\lam\geq \mu$ ($\lam>\mu$)
if $\lam-\mu$ is a (nonzero) sum of positive roots.

\begin{defn} Let $\rho\colon\g\to\gl(V)$ be a representation and $\ml$ a
minimal Levi subalgebra with split Cartan subalgebra $\a$. For $\lam\in \a^*$,
the \emph{$\lam$-weight space} is the simultaneous eigenspace $V_\lam=\{v\in
V\st \forall\,h\in\a,\;\rho(h)v=\lam(h) v\}$, which is a representation of
$\ml$.  If $V_\lam\neq 0$, $\lam$ is called a \emph{weight} of $V$. Let $\pb$
be a minimal parabolic containing $\ml$. A \emph{lowest weight vector} with
\emph{lowest weight} $\lam$ is a vector $v\in V_\lam$ with
$\rho(\nil(\pb),v)=0$. If $V$ is generated by a lowest weight vector, it is
called a \emph{lowest weight representation}.
\end{defn}
\begin{prop}\label{p:lwr} For $\ml\sub\pb\sub\g$ as above, let $V$ be a
lowest weight representation with lowest weight $\lam$. Then $V$ is a direct
sum of weight spaces with weights $\mu\geq\lam$, the lowest weight $\lam$ is
unique, and $V_\lam=\rho(\pb,v)=\rho(\ml,v)$ for any lowest weight vector $v$.
\end{prop}
\begin{proof} Let $v$ be a lowest weight vector and let $W$ be the span of
elements of the form $\rho(y_1)\cdots\rho(y_k)v$ with each $y_j$ either in
$\ml$ or $\g_\alpha$ with $\alpha\in\Phi^+$. A standard inductive argument
using $\rho(x)\rho(y)=\rho([x,y])+\rho(y)\rho(x)$ shows (with $x\in\pb_{-1}$)
that $W$ is $\rho(\g)$-invariant, hence $W=V$, and (with $x\in\a$) that $W$ is
a sum of weight spaces $V_\mu$ with $\mu\geq\lam$, where equality only holds
if all $y_j$ belong to $\ml$.
\end{proof}

\begin{thm}\label{t:std-pars} Let $\pb$ be a minimal parabolic subalgebra
of $\g$ containing a minimal Levi subalgebra $\ml$, so that
$\Phi=\Phi^+\sqcup\Phi^-=\bigsqcup_{j\in\Z}\Phi^j$ and $\g=\Dsum_{j\in\Z} \pb_j$.
\begin{numlist}
\item Any parabolic $\q$ containing $\pb$ is the stabilizer of a
  $1$-dimensional lowest weight space $L_\q=V_{\q}$ \textup(for $\pb$\textup)
  in a lowest weight representation $V_\q$ of $\g$, and $\lam(h_\alpha)=0$ if
  $\alpha(\xi_\q)=0$.
\item For any $\alpha\in\Phi^1$, $\g_\alpha$ is irreducible for $\ml$, and
  there is a maximal proper parabolic subalgebra $\q^\alpha$ such that for all
  $\beta\in\Phi^1$, $\beta(\xi^\alpha)=\delta_{\alpha\beta}$, where
  $\xi^\alpha=\xi_{\q_\alpha}$. If $L_{\q^\alpha}\sub V_{\q^\alpha}$ as in
  \textup{(1)}, then the lowest weight is a negative multiple of
  $\lam^\alpha\in\a^*$ with $\lam^\alpha(h_\beta)=\delta_{\alpha\beta}$.
\item $\Phi^1$ is a basis for $\Lam_r$, and parabolic subalgebras containing
  $\pb$ are in bijection with subsets $J$ of $\Phi^1$, where the parabolic
  $\q_J\leq\g$ corresponding to $J$ is $\bigcap_{\alpha\in J} \q^\alpha$ and
  $\xi_{\q_J}=\sum_{\alpha\in J}\xi^\alpha$.
\end{numlist}
\end{thm}
\begin{proof}[\pfin{Bou:gal,Hum:lar,Pro:lg}] (1) Let $L_\q:=\Wedge^d\nil(\q)
\leq \Wedge^d\g$ with $d=\dimn\nil(\q)$ so $\dimn L_\q=1$, and let
$V_\q\sub\Wedge^d\g$ be the $\g$-submodule generated by $L_\q$. Since
$L_\q\sub V_\q$ has stabilizer $\q$, $\nil(\pb)\leq [\q,\q]$ acts trivially
and $L_\q$ is a lowest weight space for $\pb$. If $\alpha(\xi_\q)=0$ then
$\g_{\pm\alpha}\sub\q$ and hence $h_\alpha\in[\q,\q]$.

(2)--(3) Since $\hpb=\sum_{j\in\N} \pb_j$ is a parabolic (opposite to $\pb$),
$\pb_1$ generates $\pb_+=\nil(\hpb)=\sum_{j\in\Z^+}\pb_j
=\sum_{\alpha\in\Phi^+} \g_\alpha$---hence $\Phi^1$ spans $\Lam_r$ and for any
proper (parabolic) subalgebra $\q\leq\g$ containing $\pb$, $\q\cap\pb_1$ is a
proper $\ml$-invariant subspace of $\pb_1$.

If $\s$ is a maximal proper $\ml$-invariant subspace of $\pb_1$, and $\n$ is
the (nilpotent) Lie subalgebra of $\pb_+$ generated by $\s$, then an inductive
argument shows that $[\pb,\n]\sub\pb\dsum\n$, so that $\q:=\pb\dsum\n$ is
parabolic in $\g$ with $\q\cap\pb_1=\s$. The positive eigenspaces of $\xi_\q$
meet $\pb_1$ in an irreducible complement to $\s$, which must be a root space
$\g_\alpha$ with $\alpha\in\Phi^1$. Hence all such $\s$ have the form
$\s=\Dsum\{\g_\beta\st \beta\in\Phi^1,\beta\neq\alpha\}$ and the corresponding
$\q=\q^\alpha$ has $\beta(\xi_{\q})=\delta_{\alpha\beta}$ for all
$\beta\in\Phi^1$. The rest of the theorem follows straightforwardly.
\end{proof}

Thus $\Phi^1$ is a \emph{basis of simple roots}, \ie, a $\Z$-basis for
$\Lam_r$ with respect to which any $\alpha\in\Phi\sub\Lam_r$ either has all
coefficients nonnegative, or all coefficients nonpositive. The corresponding
\emph{fundamental coweights} $\xi^\alpha:\alpha\in\Phi^1$ form a basis for
$\a\cap[\g,\g]$ so that $\rank\Lam_r=\dimn(\a\cap[\g,\g])$. The weights
$\lam^\alpha\in(\a\cap[\g,\g])^*$ with
$\lam^\alpha(h_\beta)=\delta_{\alpha\beta}$ are also uniquely determined, and
called the \emph{fundamental weights}.

Conversely, any basis of simple roots $\Psi\sub\Phi$ determines a unique
element $\xi$ of $\Lam_{cw}$ with $\alpha(\xi)=1$ for all $\alpha\in\Psi$, and
then $\p_\xi$ is a minimal parabolic subalgebra of $\g$ with $\Psi=\Phi^1$.

\section{Global theory and parabolic buildings}\label{s:gta}

We now have sufficient information to adopt a more global perspective on
parabolic subalgebras of a reductive Lie algebra $\g$.  Let $G$ be a connected
algebraic group with Lie algebra $\g$ (over a field $\F$ of characteristic
zero), so that the centre $Z(G)$ has Lie algebra $\z(\g)$, and $G/Z(G)\cong
G^{ss}\leq\Aut(\g)$ is the identity component of the automorphism group,
called the \emph{adjoint group} of $\g$; $G^{ss}$ is generated by $\exp(\ad
x)$ for $x\in \cN(\g)$ (these generate a connected normal subgroup of
$\Aut(\g)$ whose Lie algebra meets every simple component of $\g$
nontrivially).  For a subspace $\s\sub\g$ we let $N_G(\s)\leq G$ be the
stabilizer of $\s$ and $C_G(\s)$ the kernel of the action of $N_G(\s)$ on
$\s$.  These subgroups have Lie algebras $\n_{\g}(\s)$ and $\c_{\g}(\s)$
respectively.

\subsection{Homogeneity and generalized flag varieties}

Via its adjoint group, $G$ acts on the set $\PF$ of parabolic subalgebras
$\q$ of $\g$, and its orbits are called \emph{generalized flag varieties}.
Since any such $\q$ is self-normalizing ($\n_{\g}(\q)=\q$), its stabilizer
$N_G(\q)$ has Lie algebra $\q$.

\begin{prop} Any generalized flag variety embeds into the projective space
of a lowest weight representation of $\g$.
\end{prop}
\begin{proof} Use Theorem~\ref{t:std-pars} (1): for any $g\in G$, the
infinitesimal stabilizer of $g\act L_\q=L_{g\act\q}$ is $g\act\q$ and hence
the adjoint orbit of $\q$ is isomorphic to the $G$-orbit of $L_\q\in
\Proj(V_\q)$.
\end{proof}

\begin{prop}\label{p:rs} For $\alpha\in\Phi$, let $h_\alpha=[x_\alpha,y_\alpha]\in
\a$ with $\alpha(h_\alpha)=2$, $x_\alpha\in\g_\alpha$ and $y_\alpha\in\g_{-\alpha}$.
\begin{numlist}
\item The automorphism $g= \exp(\ad(x_\alpha))\exp(\ad(-y_\alpha))
  \exp(\ad(x_\alpha))$ of $\g$ preserves $\a\cap[\g,\g]$, sending $h$ to
  $g\act h=h - \alpha(h) h_\alpha$\textup;
\item For all $\alpha,\beta\in\Phi$, $\beta(h_\alpha)\in\Z$ and
  $\sigma_\alpha(\beta):=\beta-\beta(h_\alpha)\alpha\in \Phi$.
\end{numlist}
Finally, if $\alpha_1,\ldots \alpha_k$ is a basis of simple roots, and
$\Phi_+$ the corresponding set of positive roots, then for any $i\in\{1,\ldots
k\}$, $\sigma_{\alpha_i}$ permutes $\Phi_+\setminus\spn\{\alpha_i\}$.
\end{prop}
\begin{proof}[\pfin{Bou:gal,Hum:lar}]  (1) The automorphism $g=
\exp(\ad(x_\alpha))\exp(\ad(-y_\alpha))\exp(\ad(x_\alpha))$ restricts to the
identity on $\kernel\alpha$ and sends $h_\alpha$ to $-h_\alpha$.  It thus
sends any $h\in\a$ to $h - \alpha(h) h_\alpha\in \a$ as required.

(2) For any $\alpha\in\Phi$, such a triple $x_\alpha,y_\alpha,h_\alpha$ exists
and spans an $\sgl_2$ subalgebra of $\g$.  Since $\beta(g\act
h)=\beta(h-\alpha(h)h_\alpha)=(\beta-\beta(h_\alpha)\alpha)(h)
=\sigma_\alpha(\beta)(h)$, and $g\cdot[h,z] =[g\cdot h,g\cdot z]$, $g$
restricts to an isomorphism $\g_\beta\cong\g_{\sigma_\alpha^{-1}(\beta)}
=\g_{\sigma_\alpha(\beta)}$.  Hence if $\beta\in\Phi$,
$\sigma_\alpha(\beta)\in\Phi$. The $\sgl_2$ relations $\ad h_\alpha\circ \ad
x_\alpha= \ad x_\alpha\circ(\ad h_\alpha +2)$ and $\ad h_\alpha\circ \ad
y_\alpha= \ad y_\alpha\circ(\ad h_\alpha-2)$ then show that the eigenvalue
$\sigma_\alpha(\beta)(h_\alpha)=-\beta(h_\alpha)$ differs from
$\beta(h_\alpha)$ by $2k\in 2\Z$, \ie, $\beta(h_\alpha)\in\Z$.

The last part is standard: any $\alpha=\sum_{j=1}^k n_j \alpha_j\in
\Phi_+\setminus\spn\{\alpha_i\}$ has $n_j>0$ for some $j\neq i$, as does
$\sigma_{\alpha_i}(\alpha)=\alpha-\alpha(h_i)\alpha_i$; hence
$\sigma_{\alpha_i}(\alpha)$ is positive.
\end{proof}

This result shows that $\Phi$ is a \emph{root system} in $\Lam_r$, and we
refer to $\sigma_\alpha$, for $\alpha\in\Phi$, as a \emph{root reflection}.
The system need not be ``reduced'': if $\alpha$ is a root, then there may be
integer multiples of $\alpha$ other than $\pm\alpha$ which are roots; however,
$\sigma_{m\alpha}=\sigma_{\alpha}$ for any $m\in\Z\setminus\{0\}$. There are
irreducible nonreduced systems denoted $BC_n$ in addition to the Dynkin
classification~\cite{Hum:lar}.

\begin{prop} $N_G(\ml)$ acts transitively on minimal parabolic subalgebras
containing $\ml$.
\end{prop}
\begin{proof} To show any two minimal parabolic subalgebras $\pb,\pc$
containing $\ml$ are conjugate by an element of $N_G(\ml)$, let us say a root
$\alpha$ is ``shared'' if the multiples $\alpha$ which are positive for $\pb$
are also positive for $\pc$; otherwise, they are negative for $\pc$ and we say
$\alpha$ is ``unshared''. We now use complete induction on the number of
unshared roots.

If there are none then $\pb=\pc$, otherwise there is a root space
$\g_{\alpha_i}$ in $\pb_1\cap\pc$, so $\alpha_i$ (and its multiples) is
unshared.  By Proposition~\ref{p:rs}, there exists $g\in N_G(\a)$ such that
$g\act\pb$ has positive roots $\sigma_{\alpha_i}(\Phi_+)$, where $\Phi_+$ is
the set of positive roots for $\pb$; now $g\act\pb$ and $\pc$ have fewer
unshared roots, hence are conjugate under $N_G(\ml)$.
\end{proof}

Note that the proof shows more: if we consider $\pc$ to be fixed, then
the element of $N_G(\a)$ needed to conjugate $\pb$ to $\pc$ is obtained by
an iterative application of simple root reflections.

\subsection{The global Weyl group and parabolic incidence system}

Let $\bc$ be the set of all minimal parabolic subalgebras of $\g$, let $\bk$
be the set of all ``\BK/-pairs'' $(\pb,\ml)$ where $\pb$ is a minimal
parabolic subalgebra of $\g$ and $\ml$ is a Levi subalgebra of $\pb$, and let
$\Ac$ be the set of minimal Levi subalgebras in $\g$.  Given $\ml\in\Ac$, let
$\a\leq\z(\ml)$ be the corresponding split Cartan subalgebra, so that
$\ml=\c_{\g}(\a)=\n_{\g}(\a)$ is the Lie algebra of $\ML:=C_G(\a)$.

\begin{thm}\label{t:bk-act} $G$ acts transitively on $\bk$, hence also
on $\bc$ and $\Ac$.  The stabilizer of $(\pb,\ml)\in\bk$ is $\ML=C_G(\a)$.
\end{thm}
\begin{proof} Any two minimal parabolic subalgebras $\pb,\pc\leq\g$
share a (minimal) Levi subalgebra $\ml$, hence are conjugate by an element of
$N_G(\ml)\leq G$. On the other hand, the set of such $\ml$ contained in a
given $\pb$ is a torsor for $\exp(\nil(\pb))\leq G$. Now suppose that $g\in G$
fixes $(\pb,\ml)$ and let $\lam_1,\ldots\lam_k\in(\a\cap[\g,\g])^*$ be the
fundamental weights corresponding to $\pb$.  For each $j$ there is an
irreducible lowest weight representation $V_j$ whose lowest weight is a
positive multiple of $\lam_j$. Since the action of $g$ on $V_j$ sends weight
spaces to weight spaces and lowest weight vectors to lowest weight vectors,
the induced action on $\a\cap[\g,\g]$ fixes $\lam_j$. Since
$\a=\z(\g)\dsum(\a\cap[\g,\g])$ and the fundamental weights span
$(\a\cap[\g,\g])^*$, $g\in C_G(\a)$.
\end{proof}

\begin{prop}\label{p:std-unique} Let $\cQ$ be an adjoint orbit of maximal
\textup(proper\textup) parabolic subalgebras of $\g$.  Then for any minimal
parabolic subalgebra $\pb$, there is a unique $\q\in\cQ$ with $\pb\leq\q$.
\end{prop}
\begin{proof} Fix a Levi subalgebra $\ml\leq\pb$. Since $G$ acts
transitively on minimal parabolic subalgebras, for any $\qq\in\cQ$, there
exists $g\in G$ with $g\act\pb\sub\qq$, so $\pb\sub\q:=g^{-1}\act\qq\in\cQ$.
Suppose now that $\q,\qq=g\act\q\in \cQ$ both contain $\pb$; then $\q$
contains both $\b$ and $g^{-1}\act\b$, which are therefore conjugate by an
element of $N_G(\q)$ (since $N_G(\q)$ acts transitively on minimal parabolic
subalgebras of $\q/\q^\perp$). Thus we may assume $\qq=g\act\q$ with $g\in
N_G(\pb)$, and fix a (minimal) Levi subalgebra $\ml$ of $\pb$. Now since $\ml$
and $g^{-1}\act\ml$ are both Levi subalgebras of $\pb$, they are related by an
element of $\exp(\pb^\perp)\sub N_G(\pb)\sub N_G(\q)$ so we may assume $g\in
N_G(\ml)\cap N_G(\pb)=C_G(\a)$ by Theorem~\ref{t:bk-act}. Hence $\qq=\q$.
\end{proof}
Thus costandard maximal parabolic subalgebras in the same adjoint orbit are
equal.

We summarize the development so far with a double fibration of $G$-homogeneous
spaces
\begin{diagram}[size=1em]
&&\bk&&\\
&\ldTo &&\rdTo\\
\Ac&&&&\bc
\end{diagram}
in which the fibre over $\pb\in\bc$ is an $\exp(\nil(\pb))$-torsor (the Weyl
structures for $\bc$). If we choose a basepoint $(\pb,\ml)\in\bk$, then the
orbit-stabilizer theorem provides isomorphisms $\bk\cong G/\ML$, $\bc\cong
G/N_G(\pb)$ and $\Ac\cong G/N_G(\ml)$.  The fibre of $\bk$ over $\ml\in\Ac$ is
a torsor for the \emph{local Weyl group}
$W_\ml(\g):=N_G(\a)/C_G(\a)=N_G(\ml)/\ML$.

\begin{defn} The \emph{global Weyl group} $W(\g)$ of $\g$ is the automorphism
group $\Aut_G(\bk)$ of the \emph{Weyl space} $\bk$ of \BK/-pairs $(\pb,\ml)$,
\ie, $W(\g$) is the set of bijections $\bk\to\bk$ commuting with the $G$-action.
\end{defn}
We shall write the $G$-action on $\bk$ on the left, and the $W(\g)$ action on
the right.
\begin{prop} The Weyl space $\bk$ is a principal $W(\g)$-bundle over $\Ac$,
\ie, the action of $W(\g)$ on $\bk$ is fibre-preserving, and each fibre is a
$(W_\ml(\g), W(\g))$-bitorsor. In particular, any basepoint $(\pb,\ml)\in\bk$
yields an isomorphism between $W_\ml(\g)$ and $W(\g)$.
\end{prop}
\begin{proof} Note that $W(\g)$ preserves the fibration of $\bk$ over $\Ac$ and
thus induces a right action of $W(\g)$ on $\Ac$ commuting with $G$. However,
any $\ml\in \Ac$ is the Lie algebra of its stabilizer $N_G(\ml)$ in $G$, and
for any $w\in W(\g)$, $N_G(\ml w)=N_G(\ml)$, so $\ml w=\ml$. Thus the induced
action of $W(\g)$ on $\Ac$ is trivial, \ie, the action of $W(\g)$ on $\bk$ is
fibre-preserving. Since the $G$-action is transitive, the $W(\g)$ action is
free, and so any $w\in W$ is uniquely determined by what it does to a
base-point $(\pb,\ml)\in\bk_{\;\;\;\ml}$. If also
$(\pc,\ml)\in\bk_{\;\;\;\ml}$, there is an automorphism sending
$g\act(\pb,\ml)$ to $g\act(\pc,\ml)$ for all $g\in G$, because the effective
quotient $W_\ml(\g)=N_G(\ml)/\ML$ acts freely on $\bk_{\;\;\;\ml}$. Hence
$W(\g)$ also acts transitively on fibres.
\end{proof}

\begin{rem} Another description of $W(\g)$ uses a natural groupoid structure
on $\Ac$: the set of morphisms from $\ml$ to $\mlp$ in $\Ac$ is
$N_G(\ml,\mlp)/\ML$, \ie, the set of elements of $G$ which conjugate $\ml$ to
$\mlp$, modulo the right action of $\ML$ (or equivalently, the left action of
$\ML'$). Since $\Hom_{\Ac[]}(\ml,\ml)=W_\g(\ml)$, we refer to $\Ac$, with this
structure, as the \emph{Weyl groupoid of $\g$}.

Now $\bk$ determines a representation (or action) of the Weyl groupoid $\Ac$:
any $[g]\in\Hom_{\Ac[]}(\ml,\mlp)=N_G(\ml,\mlp)/\ML$ induces a function
$\bk_{\;\;\;\ml}\to\bk_{\;\;\;\mlp}$ sending $(\pb,\ml)$ to $(g\act\pb,
g\act\ml)=(g\act\pb,\mlp)$; these determine a functor $\Ac\to\mathbf{Set}$.
Furthermore, on each fibre $\bk_{\;\;\;\ml}$, the action of $W_\g(\ml)$ is
free and transitive. Thus we may equip $\bk$ with a trivial (2-connected)
groupoid structure over $\Ac$, in which there is a unique morphism between any
two points $(\pb,\ml)$ and $(\pc,\mlp)$, labelled by the unique element of
$\Hom_{\Ac[]}(\ml,\mlp)= N_G(\ml,\mlp)/\ML$ whose representatives send $\pb$
to $\pc$. In other words, $\bk$ is a \emph{universal groupoid cover} of $\Ac$,
and $W(\g)$ is its group of ``deck transformations''.
\end{rem}

In addition to its $G$-space structure, $\PF$ has an incidence relation:
$\p,\pp\in \PF$ are incident if they are costandard. This incidence structure
of $\PF$ is determined by the maximal proper parabolic subalgebras $\p$.  Let
$\typ_\g$ be the set of adjoint orbits of such $\p$.

\begin{defn} The \emph{parabolic incidence system} $\PI{\g}$ of $\g$ is the
set of all maximal proper parabolic subalgebras $\p\leq\g$, equipped with the
incidence relation
\begin{equation*}
\p\edge{}\pp\quad\text{iff}\quad \text{$\p$ and $\pp$ are costandard, \ie,
  $\p\cap\pp$ is parabolic in $\g$}
\end{equation*}
and the type function $t_{\PI{\g}}\colon|\PI{\g}|\to\typ_\g$ sending $\p$ to
its adjoint orbit $G\act\p$.
\end{defn}
\begin{thm} \label{t:par-flag-geom} Let $\g$ be a reductive Lie algebra. Then
the map
\begin{equation*}
\Lam\colon\bigsqcup_{J\sub\typ_\g}\cF\PI{\g}(J)\to\PF;\;\sigma\mapsto
\bigcap_{\p\in\sigma}\p
\end{equation*}
induces an incidence isomorphism over $\Pow(\typ_\g)$, where parabolics
$\p,\q\in \PF$ are incident if they are costandard, and the fibres of the type
function $\PF\to\Pow(\typ_\g)$ are adjoint orbits.
\end{thm}
\begin{proof} By Proposition~\ref{p:costd-pars}, $\Lam$ is well-defined. For
any $\q\in\PF$, after choosing a minimal parabolic subalgebra $\pb$ contained
on $\q$, Theorem~\ref{t:std-pars} and Proposition~\ref{p:std-unique} show
that the maximal parabolic subalgebras containing $\q$ have intersection $\q$
and form a $J$-flag in $\cF\PI{\g}(J)$ for some $J\in \Pow(\typ_\g)$. Now if
$\q=\p^1\cap\cdots\cap \p^k$ and $g\in G$ then $g\act\q= g\act\p^1\cap\cdots
g\act\p^k$, so the adjoint orbit of $\q$ is in the image of $\cF\PI{\g}(J)$
and we may label it by $J$. Distinct adjoint orbits have distinct labels by
Proposition~\ref{p:std-unique} and $\Lam$ intertwines incidence of flags
with incidence of parabolic subalgebras. Hence it is an incidence isomorphism.
\end{proof}
Henceforth we label parabolic adjoint orbits (\ie, generalized flag varieties)
$\PF(J)$ in $\g$ by subsets $J$ of $\typ_\g$ using this isomorphism.

\begin{example}\label{ex:4A} Continuing Examples~\ref{ex:1A}--\ref{ex:3A},
we take $G=\GL(V)$, and the adjoint group is $\PGL(V)$. There is a bijection
between minimal Levi subalgebras $\ml\leq\gl(V)$ and $n+1$ element subsets
$\cS_\ml\sub\Proj(V)=\Gr_1(V)$ with $V=\Dsum \cS_\ml$:
$\cS_\ml=\{L\in\Proj(V)\st \ml\cdot L\sub L\}$ and $\ml=\bigcap_{L\in\cS_\ml}
\stab_{\gl(V)}(L)$, which is the Cartan subalgebra of diagonal matrices with
respect to any basis representing the elements of $\cS_\ml$, hence $\a=\ml$.
Thus $N_G(\ml)$ is the subgroup of $\GL(V)$ preserving the decomposition
$V=\Dsum \cS_\ml$ and the local Weyl group $W_{\ml}(\gl(V))=N_G(\ml)/C_G(\ml)$
is canonically isomorphic to $\Sym(\cS_\ml)$.

The minimal parabolic subalgebra $\pb$ stabilizing the full flag $0=W_0\leq
W_1\leq\cdots \leq W_n\leq W_{n+1}=V$ contains $\ml$ if and only if there is a
bijection $j\mapsto L_j\colon\{1,\ldots n+1\}\to\cS_\ml$ such that
$W_j=W_{j-1}\dsum L_j$, and hence $W_j=L_1\dsum\cdots\dsum L_j$. Hence the
fibre of $\bk$ over $\ml\in\Ac$ may be identified canonically with the set of
bijections $\{1,\ldots n+1\}\to \cS_{\ml}$, and the global Weyl group $W(\g)$
with $\Sym_{n+1}=\Sym(\{1,\ldots n+1\})$. Note that for each $j\in\{1,\ldots
n\}$ the transposition $\sigma_j=(j\ j+1)$ sends $\pb$ to the stabilizer of
the full flag $\tilde W_1\leq\cdots\leq \tilde W_n$ with $\tilde W_i=W_i$ for
$i\neq j$ and $\tilde W_j=W_{j-1}\dsum L_{j+1}$. These transpositions generate
$\Sym_{n+1}$.

A maximal (proper) parabolic subalgebra contains $\ml$ if and only if the
subspace it stabilizes is a sum of elements of $\cS_\ml$. Hence the incidence
system of such parabolic subalgebras is isomorphic to the incidence system of
proper nontrivial subsets of $\cS_\ml$, cf. Example~\ref{ex:1A}.
\end{example}
\begin{example}\label{ex:4B} Continuing Examples~\ref{ex:1B}--\ref{ex:3B},
we take $G$ to be the adjoint group $\SO_0(U,Q_U)$.  Minimal Levi subalgebras
$\ml$ now correspond to orthogonal direct sum decompositions $U= \bigl(\Dsum
\cR_\ml)\dsum W_\ml$ where $Q_U$ is positive definite on $W_\ml$, $\dim
W_\ml=k$, and $\cR_\ml$ is a set of $n$ signature $(1,1)$ subspaces of $U$.
Again $N_G(\ml)=N_G(\a)$ is the subgroup of $\SO_0(U,Q_U)$ preserving this
decomposition of $U$, and the local Weyl group
$W_{\ml}(\so(U,Q_U)))=N_G(\ml)/\ML$ is isomorphic to
$\Sym(\cR_\ml)\subnormal{\Z_2}^{\cS_\ml}$, where any representative of
$f\in{\Z_2}^{\cS_\ml}$ preserves any $R\in\cR_\ml$, and if $f(R)=0$, it
preserves the two isotropic lines in $R$; otherwise, it swaps them.

The minimal parabolic subalgebra $\pb$ stabilizing the full flag
$\pi_1\leq\cdots\leq\pi_n$ of isotropic subspaces of $U$ contains $\ml$ if and
only if for each $j\in\{1,\ldots n\}$ there is an isotropic line $L_j\leq
R\in\cR_\ml$ such that $\pi_j=\pi_{j-1}\dsum L_j$ (where $\pi_0=0$). The fibre
of $\bk$ over $\ml\in\Ac$ is isomorphic to the space of such maps, which
identifies the global Weyl group with $\Sym_n\subnormal{\Z_2}^n$. The
canonical generators are $\rho_j=(j\ j+1)$ for $j\in\{1,\ldots n-1\}$, while
$\rho_n\in{\Z_2}^n$ with $\rho_n(j)=\delta_{jn}$. Thus for any $j\in\{1,\ldots
n\}$, $\rho_j$ sends $\pb$ to the stabilizer of
$\tilde\pi_1\leq\cdots\leq\tilde\pi_n$ where $\tilde\pi_i=\pi_i$ for $i\neq
j$, $\tilde\pi_j=\pi_{j-1}\oplus L_{j+1}$ for $j\neq n$, and otherwise, if
$j=n$, $\tilde\pi_n=\pi_{n-1}\dsum\tilde L_n$ where $\tilde L_n$ is isotropic
with $L_n\dsum\tilde L_n\in\cR_\ml$.

Let $\cR$ be a set of $n$ disjoint two element sets. A subset of $\bigcup\cR$
is \emph{admissible} if it contains at most one element from each two element
set. The nonempty admissible subsets form an incidence system
$\grph^{\cR,\pm}$ over $\cI_n$, with the number of elements as the type, and
incidence by containment. The incidence system of maximal parabolic
subalgebras of $\so(U,Q_U)$ containing $\ml$ is isomorphic to
$\grph^{\cR,\pm}$; this is the discrete model promised in Example~\ref{ex:1B}.
\end{example}

In general, the preceding analysis of parabolic subalgebras can be used to
show that $\PF$ is a Tits building~\cite{AbBr:bta,Gar:bcg,Tits:bst} with
apartment complex $\Ac$.  We find it more convenient to do this in the
framework of chamber \grsys/s~\cite{BuCo:dg,Ron:lb,Wei:ssb}, initiated by Tits
in~\cite{Tits:lab}.

\subsection{Chamber \grsys/s and Coxeter groups}

\begin{defn}[\cite{Ron:lb,Wei:ssb}] A \emph{chamber \grsys/} over $\typ$ is a
graph $\cgph$ with an edge labelling $\lam\colon E_\cgph\to \typ$ such that
for each $i\in\typ$, the \emph{$i$-adjacency} relation $b\edge{i}c$ (\ie,
$b=c$ or $\{b,c\}$ is an edge with label $i$) is an equivalence relation. A
\emph{chamber sub\grsys/} is a subgraph with the induced edge labelling, and
a \emph{chamber morphism} $\cm\colon\cgph\to\cgph'$ over $\typ$ is morphism of
edge labelled graphs, \ie, a map on vertices such that $b\edge{i}c$ implies
$\cm(b)\edge{i}\cm(c)$.
\end{defn}
The vertices of a chamber \grsys/ are called \emph{chambers}, and the
equivalence class of a chamber $c$ under $i$-adjacency is called its
\emph{$i$-panel} $p_i(c)$. We shall often blur the distinction between $\cgph$
and its underlying set $|\cgph|$ of chambers. For an $i$-panel $p$ and a
chamber sub\grsys/ $A$, we say $p\in A$ if $A\cap p$ is nonempty (\ie, for
some chamber $c\in A$, $p=p_i(c)$).

\begin{prop}\label{p:is-cs} Let $\grph$ be an incidence system over $\typ$.
Then the set of full flags in $\grph$ can be made into a chamber \grsys/
$\cgph=\cC\grph$ via the relation $b\edge{i}c$ iff their subflags of type
$\typ\setminus\{i\}$ are equal: $b\cap t_\grph^{\;-1}(\typ\setminus\{i\})
=c\cap t_\grph^{\;-1}(\typ\setminus\{i\})$.
\end{prop}

\begin{defn} A \emph{firm} $i$-panel is one containing at least two chambers;
it is \emph{thin} if it has exactly two, and \emph{thick} if it has more than
two. A firm, thin or thick chamber \grsys/ is one whose $i$-panels (for all
$i\in\typ$) are all firm, thin or thick respectively.
\end{defn}

In a thin chamber \grsys/ $\cgph$, for each $i\in\typ$ there is a free
involution of $|\cgph|$, which interchanges the two chambers in each
$i$-panel. These involutions are not chamber automorphisms; they generate a
right action on $|\cgph|$ of the free group over $\typ$ with the property that
$b\edge{i}c$ iff $b=ci$ iff $bi=c$.

\begin{defn} The \emph{structure group} of a thin chamber \grsys/ $\cgph$
is the subgroup of $\Sym(|\cgph|)$ generated by the involutions $b\mapsto bi$
for $i\in \typ$ (the effective quotient of the free group action generated by
$\typ$). Its group elements are the vertices of a chamber \grsys/ with
$w\edge{i}w'$ iff $w'=wi$ (the \emph{Cayley graph} of the presentation) and we
denote the structure group (viewed in this way) by $W_\cgph$, or by $\typ\into
W_\cgph$ (to indicate the generating set).
\end{defn}
\begin{prop} Let $\cgph$ be a thin chamber \grsys/.
\begin{bulletlist}
\item A permutation of $|\cgph|$ is a chamber morphism if and only if
  it commutes with $W_\cgph$.
\item $\cgph$ is connected if and only if $W_\cgph$ acts transitively. In this
  case the automorphism group $\Aut(\cgph)$ acts freely and the following are
  equivalent\textup:
\begin{numlist}
\item for any adjacent $b,c\in \cgph$, there is an automorphism of $\cgph$
interchanging $b$ and $c$\textup;
\item $\cgph$ is homogeneous, \ie, $\Aut(\cgph)$ acts transitively on
\textup(the chambers of\textup) $\cgph$\textup;
\item $W_\cgph$ acts freely on $\cgph$\textup;
\item there is a unique function $\delta_\cgph\colon \cgph\times\cgph \to
  W_\cgph$ with $\delta_\cgph(b,c)=w$ iff $c=bw$.
\end{numlist}
\end{bulletlist}
\end{prop}
Thus a connected homogeneous thin chamber \grsys/ $\cgph$ is an
$(\Aut(\cgph),W_\cgph)$-bitorsor, and, for fixed $b\in\cgph$, the map
$\delta_\cgph(b,\cdot)\colon \cgph\to W_\cgph$ is an isomorphism of chamber
\grsys/s.

\subsection{Apartments and buildings} Henceforth, all chamber \grsys/s
will be nonempty, firm and connected.

\begin{defn} Let $\cgph$ be a chamber \grsys/ and let $\typ\into W$ be a group
generated by involutions. An \emph{apartment} in $\cgph$ is a
\textup(connected\textup) homogeneous thin chamber sub\grsys/ $A$ of
$\cgph$. A \emph{$W$-distance} on $\cgph$ is a map $\delta_\cgph
\colon\cgph\times\cgph\to W$ such that for any $b,c,c'\in\cgph$,
$\delta_\cgph(b,c)=\delta_\cgph(c,b)^{-1}$ and $c'\edge{i} c$ in $\cgph$
implies $\delta_\cgph(b,c')\edge{i}\delta_\cgph(b,c)$ in $W$.  An apartment
$A$ is \emph{compatible} with a $W$-distance $\delta_\cgph$ if $W_A\cong W$
with $\delta_A=\delta_\cgph\restr{A}$.  We say $(\cgph,\delta_\cgph)$ is a
\emph{building} (of type $W$) if for any $b,c\in \cgph$ there is a compatible
apartment $A$ containing $b$ and $c$.
\end{defn}
A thin building $(\cgph,\delta_\cgph)$ is just a homogeneous thin chamber
\grsys/ with its natural function $\delta_\cgph\colon \cgph\times\cgph\to
W=W_\cgph$.  Strictly speaking, we should require that $W$ is a \emph{Coxeter
  group} (as it will be in our examples). In the literature, the notion of a
building is usually defined either purely in terms of $\delta_\cgph$ (with no
mention of apartments) or in terms of apartments (often using the language of
simplicial complexes rather than chamber \grsys/s). We now relate our hybrid
approach to the latter.

The idea is that we can define $\delta_\cgph(b,c)=\delta_A(b,c)$ for an
apartment $A$ containing $b$ and $c$, provided that the apartments we admit
give the same answer, and yield a $W$-distance.

\begin{lemma}\label{l:reg-complex} Let $\cgph$ a chamber \grsys/ and
let $\cm\colon A_1\to A_2$ be an isomorphism between apartments $A_1$ and
$A_2$ fixing $b\in A_1\cap A_2$. Then\textup:
\begin{numlist}
\item $\cm$ fixes $A_1\cap A_2$ pointwise if and only if for any $c\in A_1\cap
  A_2$, $\delta_{A_1}(b,c)=\delta_{A_2}(b,c)$\textup;
\item for any $c_1\in A_1$ and $c_2\in A_2$ with $c_1\edge{i} c_2$, \ie,
  $p:=p_i(c_2)=p_i(c_2)$, we have that
  $\delta_{A_1}(b,c_1)\edge{i}\delta_{A_2}(b,c_2)$ if and only if $\cm$ sends
  $p\cap A_1$ to $p\cap A_2$.
\end{numlist}
\end{lemma}
\begin{proof} (1) If $c=b w\in A_1$ then $\cm(c)=b w$ in $A_2$, which is $c$
iff $\delta_{A_2}(b,c)=w=\delta_{A_2}(b,c)$.

(2) If $w=\delta_{A_1}(b,c_1)$, then $\cm$ sends $p\cap A_1=\{c_1,c_1i\}$ to
$\{c',c'i\}\sub A_2$ for the unique $c'\in A_2$ with
$\delta_{A_2}(b,c')=w$. This is $p\cap A_2=\{c_2,c_2 i\}$ if and only if
$\delta_{A_1}(b,c_1)\edge{i}\delta_{A_2}(b,c_2)$.
\end{proof}
In (2), we say that $A_1$ and $A_2$ \emph{share} the $i$-panel $p$, that
$p\in A_1$ and $p\in A_2$, and that an isomorphism $\cm$ sending $p\cap A_1$
to $p\cap A_2$ \emph{preserves} $p$.
\begin{defn} An \emph{apartment complex} $\Ac[]$ in $\cgph$ is a set of
apartments such that any two chambers in $\cgph$ belong to a common apartment
in $\Ac[]$.  We say $\Ac[]$ is \emph{regular} iff for any two intersecting
apartments $A_1,A_2\in \Ac[]$, there is a chamber isomorphism $A_1\to A_2$
fixing $A_1\cap A_2$ pointwise and preserving any $i$-panel they share.
\end{defn}
\begin{prop} A chamber \grsys/ $\cgph$ is a building with respect to
some $\delta_\cgph\colon\cgph\times\cgph\to W$ if and only if it admits a
regular apartment complex $\Ac[]$ \textup(whose apartments have type
$W$\textup), in which case there is a unique $\delta_\cgph$ such that the
apartments in $\Ac[]$ are compatible.
\end{prop}
\begin{proof} If $(\cgph,\delta_\cgph)$ is a building then the set $\Ac[]$
of compatible apartments form an apartment complex, while given an apartment
complex $\Ac[]$, there is at most one $\delta_{\cgph}$ with
$\delta_\cgph\restr{A}=\delta_A$ for all $A\in\Ac[]$.  The equivalence now
follows easily from Lemma~\ref{l:reg-complex}.
\end{proof}
The chamber isomorphisms $A_1\to A_2$ often arise as restrictions of
automorphisms of $\cgph$.
\begin{defn} A group $G$ of chamber automorphisms of a chamber \grsys/ $\cgph$
preserving an apartment complex $\Ac[]$ is said to be \emph{strongly
  transitive} if it acts transitively on $\bk[]:=\{(a,A)\st a\in A\in\Ac[]\}$.
In this situation $\Ac[]$ is regular if and only if
\begin{itemize}
\item[(R)] for any chamber $b\in\cgph$ and any $i$-panel or chamber $p$,
$\Stab_G(b)\cap\Stab_G(p)$ acts transitively on $\{A\in\Ac[]:b,p\in A\}$.
\end{itemize}
\end{defn}
Indeed condition (R) implies regularity by Lemma~\ref{l:reg-complex}.
Conversely, if $\Ac[]$ is regular, then for any $A_1,A_2\in\Ac[]$ and any
$b\in A_1\cap A_2$, the unique isomorphism $A_1\to A_2$ fixing $b$ is the
restriction of some $g\in G$ by strong transitivity, so condition (R) holds.

We can now summarize what we have established in the homogeneous case.

\begin{thm}[Abstract Bruhat Decomposition]\label{t:abstract-bruhat}
Suppose $\cgph$ is a chamber \grsys/ over $\typ\into W$ with apartment complex
$\Ac[]$ and a strongly transitive group $G$ of automorphisms satisfying
condition \textup{(R)}. Then $\cgph$ is a building, where the apartments in
$\Ac[]$ are compatible with a $W$-distance $\delta_\cgph$ which induces a
bijection $G\backslash(\cgph\times\cgph)\to W$.
\end{thm}
\begin{proof} Since $\Ac[]$ is a regular apartment complex, it determines
a $W$-distance $\delta_\cgph\colon \cgph\times\cgph\to W$ which agrees with
$\delta_A$ on any apartment $A\in\Ac[]$. Thus $\delta_\cgph$ is surjective and
$G$-invariant. Furthermore, for any $w\in W$, $G$ acts transitively on pairs
$(b,c)$ with $\delta_\cgph(b,c)=w$ (if $(b',c')$ is another such pair, then
after choosing apartments $A,A'\in\Ac[]$ containing $b,c$ and $b',c'$
respectively, the element of $G$ sending $(b,A)$ to $(b',A')$ also sends $c$
to $c'$). Hence the induced surjection $G\backslash(\cgph\times\cgph)\to W$ is
also injective.
\end{proof}

\subsection{Parabolic buildings and the Bruhat decomposition}

We now apply Proposition~\ref{p:is-cs} to the parabolic incidence system
$\PI{\g}$, to obtain a chamber \grsys/ whose chambers are the full flags of
$\PI{\g}$, which we may identify with the minimal parabolic subalgebras of
$\g$, and hence denote $\bc$; two such subalgebras $\pb,\pc$ are then
$i$-adjacent (for $i\in \typ_\g$) iff there is a (necessarily unique)
parabolic subalgebra of type $\typ_\g\setminus\{i\}$ containing both $\pb$ and
$\pc$.

\begin{thm}\label{t:bruhat} The chamber \grsys/ $\bc$ of minimal parabolic
subalgebras is a building over the global Weyl group $W(\g)$, which is a
Coxeter group, and there is a canonical bijection between
$G\backslash(\bc\times\bc)$ and $W(\g)$.
\end{thm}
\begin{proof} For any minimal Levi subalgebra $\ml\in\Ac$, the subgraph
of $\bc$ given by the minimal parabolic subalgebras containing $\ml$ is a thin
chamber \grsys/ of type $W(\g)$ whose automorphism group is the local Weyl group
$N_G(\ml)/C_G(\ml)$ (whose restricted root system action implies $W(\g)$ is a
Coxeter group).  Any two minimal parabolic subalgebras belong to such an
apartment by Corollary~\ref{c:par-reg-apt-comp}, which also shows that the
apartment complex of type $W(\g)$ defined by $\Ac$ satisfies (R).  The result
now follows from Theorem~\ref{t:abstract-bruhat}.
\end{proof}
We refer to $\bc$ as the \emph{parabolic building} of $\g$. If we fix a
minimal parabolic subalgebra $\pb$ with stabilizer $B=N_G(\pb)$ then $\bc\cong
G/B$ (as a $G$-space) and $G\backslash(\bc\times\bc)\cong
G\backslash(G/B\times G/B)\cong B\backslash G/B$, the set of double cosets of
$B$ in $G$.

\begin{prop} There is a unique and involutive map $\op_\g\colon\typ_\g\to
\typ_\g$ \textup(\ie, ${\op_\g}^2=\iden_{\typ_\g}$\textup) such that for any
$J\sub\typ_\g$ and any parabolic subalgebra $\p$ of type $J$, its opposite
parabolic subalgebras have type $\op_\g(J)$.
\end{prop}
\begin{proof} For any parabolic subalgebra $\p$, the parabolic subalgebras
opposite to $\p$ belong to a single adjoint orbit by
Remarks~\ref{r:weyl-tors}. Furthermore, if $\hp$ is opposite to $\p$
then $g\act\hp$ is opposite to $g\act\p$, and if $\pp$ is costandard with
$\p$ then by Corollary~\ref{c:weak-op-char}, there exists $\hpp$ opposite to
$\pp$ and costandard with $\hp$; it follows that $\hp\cap\hpp$ is opposite to
$\p\cap\pp$.  Hence the opposition involution of adjoint orbits is determined
uniquely by the opposition relation on maximal parabolic subalgebras.
\end{proof}
\begin{defn}
The map $\op_{\g}$ is called the \emph{duality involution} of $\g$.
\end{defn}

\section{Parabolic projection and geometric configurations}\label{s:ppgc}

\subsection{Parabolic projection}

Let $\q$ be parabolic in $\g$, let $\q_0=\q/\nil(\q)$ and let $Q=N_G(\q)\sub
G$. By Corollary~\ref{c:levi-par}, the inverse image of any parabolic in
$\q_0$ is a parabolic subalgebra of $\g$ contained in $\q$. By
Proposition~\ref{p:par-pair}, this process has a right inverse.

\begin{defn} Let $\q$ be a parabolic subalgebra of a reductive Lie algebra
$\g$. Then \emph{parabolic projection} $\PP{\q}\colon\PF\to\PF[\q_0]$ is
defined by
\begin{equation*}
\PP{\q}(\p)=(\p\cap\q+\nil(\q))/\nil(\q).
\end{equation*}
\end{defn}
The following property is immediate from the definition and the special case
$\p\sub\pp$.
\begin{prop}\label{p:PP-incidence} If $\p,\pp\leq\g$ are costandard parabolic
subalgebras then $\PP{\q}(\p\cap\pp)\sub \PP{\q}(\p)\cap\PP{\q}(\pp)$.  Thus
$\PP{\q}$ preserves incidence.
\end{prop}
By Theorem~\ref{t:par-flag-geom}, any maximal (proper) parabolic subalgebra
$\p$ of $\q$ is the intersection with $\q$ of a maximal parabolic subalgebra
of $\g$ whose adjoint orbit in $\typ_\g$ is uniquely determined by the adjoint
orbit $Q\act\p$. There is thus a natural inclusion $\iota_\q\colon
\typ_{\q_0}\to\typ_\g$ whose image is the complement of the type
$J_\q\in\Pow(\typ_\g)$ of $\q$ in $\g$. More generally, if $\p\in\PF[\q_0]$
has type $J\in \Pow(\typ_{\q_0})$ then its inverse image in $\q$ has type
$\iota_\q(J)\cup J_\q\in\Pow(\typ_\g)$.

Parabolic projection has a more complicated behaviour in general, although on
the set $\PF\cstd{\q}$ of parabolic subalgebras of $\g$ costandard with $\q$,
the behaviour is straightforward.
\begin{prop}\label{p:costd} The restriction of $\PP{\q}$ to $\PF\cstd{\q}$
sends parabolic subalgebras $\p$ of type $J\in\Pow(\typ_\g)$ to parabolic
subalgebras of type $\iota_\q^{\,-1}(J)\in\Pow(\typ_{\q_0})$.
\end{prop}
\begin{proof} If $\p$ is costandard with $\q$ then $\PP{\q}(\p)=\p\cap\q/
\nil(\q)$. Since $\p\cap\q$ is in the adjoint orbit $J\cup J_\q$, its
projection to $\q_0$ is in $\iota_\q^{\,-1}(J\setminus J_\q)=
\iota_\q^{\,-1}(J)$.
\end{proof}

Parabolic subalgebras of $\g$ are rarely costandard with $\q$; generically,
they are weakly opposite to $\q$ in the following sense.
\begin{defn}[\cite{Pan:ifs}] Parabolic subalgebras $\p,\q\leq\g$ are said to
be \emph{weakly opposite} if $\p+\q=\g$ (equivalently
$\nil(\p)\cap\nil(\q)=0$).
\end{defn}
The weakly opposite case is also amenable to analysis.
\begin{lemma}\label{c:weak-op-char} Parabolic subalgebras $\p$ and $\q$
of $\g$ are weakly opposite if and only if there is a parabolic subalgebra
$\hp$ opposite to $\p$ and costandard with $\q$.
\end{lemma}
\begin{proof} If $\hp$ is opposite to $\p$ and costandard with $\q$ then
$\nil(\q)\leq\hp$ so $\nil(\p)\cap\nil(\q)=0$. Conversely, if
$\nil(\p)\cap\nil(\q)=0$, we may take $\hp$ opposite to $\p$ using a
Weyl structure $\xi_\p$ in $\p\cap\q$: $\nil(\q)$ then has nonnegative
eigenvalues for $\xi_\p$, hence lies in $\hp$.
\end{proof}
\begin{lemma} \label{l:comp-par} Let $\p,\hp$ be parabolic in $\g$ with
$\hp$ opposite to $\p$ and costandard with $\q$. Then $\PP{\q}(\hp)$
is opposite to $\PP{\q}(\p)$ in $\q_0$.
\end{lemma}
\begin{proof} Since $\PP{\q}(\hp)=\hp\cap\q/\nil(\q)$, $\nil(\PP{\q}(\hp))
=(\nil(\hp)+\nil(\q))/\nil(\q)$.  As $\p+\nil(\hp)=\g$ with $\nil(\hp)\sub\q$,
we have $\p\cap\q+\nil(\hp)=\q$ and hence
$(\p\cap\q+\nil(\q))+(\nil(\hp)+\nil(\q))=\q$. It follows that
$\PP{\q}(\p)+\nil(\PP{\q}(\hp))=\q_0$.

It remains to show that $\PP{\q}(\hp)+\nil(\PP{\q}(\p))=\q_0$, \ie,
$\hp\cap\q +\nil(\p)\cap\q = \q$ (since $\nil(\q)\sub\hp\cap\q$). For this, we
use Corollary~\ref{c:par-reg-apt-comp} to introduce a minimal Levi
subalgebra $\ml$ of $\g$ contained in the parabolic subalgebras $\p$ and
$\hp\cap\q$. In particular, $\ml\sub\p\cap\hp$ and so the root space
decomposition of $\g$ associated to $\ml$ refines the direct sum decomposition
$\g=\hp\dsum\nil(\p)$. Since $\q$ contains $\ml$, it is a sum of root spaces,
so $\q=\g\cap\q=\hp\cap\q\dsum \nil(\p)\cap\q$.
\end{proof}

Let $\PF\wop{\q}$ be the set of all parabolic subalgebras of $\g$
which are weakly opposite to $\q$.
\begin{prop}\label{p:wop} The restriction of $\PP{\q}$ to $\PF\wop{\q}$
sends parabolic subalgebras $\p$ of type $J\in\Pow(\typ_\g)$ to parabolic
subalgebras of type $(\op_\g\circ\iota_\q\circ\op_{\q_0})^{-1}(J)
\in\Pow(\typ_{\q_0})$.
\end{prop}
\begin{proof}
This is immediate from Corollary~\ref{c:weak-op-char} and
Lemma~\ref{l:comp-par}.
\end{proof}

Minimal parabolic subalgebras also have straightforward projections.

\begin{prop} If $\pb\in\bc$ then $\PP{\q}(\pb)\in\bc[\q_0]$.  Furthermore,
if $\pb_0\in \bc[\q_0]$ then there exists $\pb\in\bc$ weakly opposite to $\q$
with $\PP{\q}(\pb)=\pb_0$.
\end{prop}
\begin{proof} Using Corollary~\ref{c:par-reg-apt-comp}, let $\ml$ be a minimal
Levi subalgebra contained in $\pb$ and $\q$, and let $\hq$ be the opposite
subalgebra to $\q$ containing $\ml$.  Then the root space decomposition of
$\ml$ refines the direct sum decomposition
$\g=\nil(\q)\dsum\q\cap\hq\dsum\nil(\hq)$, and so $\pb$ splits across this
decomposition. For any root $\alpha$ with $\g_\alpha\sub\q\cap\hq$, precisely
one of $\g_\alpha$ and $\g_{-\alpha}$ is contained in $\pb$, so that
$\pb\cap\q\cap\hq$ is minimal in $\q\cap\hq$. This is an isomorph of
$\PP{\q}(\pb)$ in $\q_0$.

For the second part, let $\pb_0=\pc/\nil(\q)$, where $\pc$ is a minimal
parabolic subalgebra of $\g$. Let $\hpc$ be a minimal parabolic subalgebra of
$\g$ opposite to $\pc$. Then $\hpc$ is weakly opposite to $\q$.  Since both
$\PP{\q}(\pc)$ and $\PP{\q}(\hpc)$ are minimal parabolic subalgebras of $\q$,
there exists $q\in Q$ such that
\[
\pc=q\act(\hpc\cap\q+\nil(\q))=(q\act\hpc)\cap\q+\nil(\q);
\]
whence $\pb:=q\act\hpc$ is weakly opposite to $\q$ with $\PP{\q}(\pb)=\pb_0$.
\end{proof}

\begin{prop} \label{p:dense-open-orbit} If $\pb,\pc\in\bc$ are both
weakly opposite to $\q$, then $\pb\in Q\cdot\pc$.
\end{prop}
\begin{proof} Using Corollary~\ref{c:weak-op-char}, let $\hpb$ and $\hpc$ be
respectively minimal parabolic subalgebras of $\g$ opposite to $\pb$ and
$\pc$, and costandard with $\q$. Then $\hpb/\nil(\q)$ and $\hpc/\nil(\q)$ are
minimal parabolic subalgebras of $\q_0$, so there exists $q\in Q$ such that
$\hpc=q\act\hpb$. It follows that $q^{-1}\act\pc$ is opposite to $\hpb$, hence
conjugate to $\pb$ by an element $q'\in\exp(\nil(\hpb))$ by
Remarks~\ref{r:weyl-tors}. Since $\nil(\hpb)\sub\q$, $\pb$ and $\pc$ are in
the same $Q$-orbit.
\end{proof}
\begin{cor} Any two parabolic subalgebras of $\g$ in the same adjoint orbit
and weakly opposite to $\q$ are in the same $Q$-orbit.
\end{cor}
\begin{proof} This follows from Theorem~\ref{t:par-flag-geom}
and Propositions~\ref{p:std-unique} and~\ref{p:dense-open-orbit}.
\end{proof}

Let $\nu_\q:=\op_\g\circ\iota_\q\circ\op_{\q_0}\colon\typ_{\q_0}\to\typ_\g$,
and observe that the direct and inverse image operations in induced by
$\nu_\q$ may be used to pull back or push forward presheaves from $\typ_\g$ to
$\typ_{\q_0}$ or vice versa. In particular,
$\nu_\q^*\PF\wop{\q}(J):=\PF\wop{\q}(\nu_\q(J))$ is isomorphic to the flag
complex of the incidence system
$\nu_{\q}^*\PI{\g}\wop{\q}:=\bigsqcup_{i\in\typ_{\q_0}} \PF\wop{\q}(\nu_\q(i))$
(with the obvious type function), while $\nu_{\q*}\PF[\q_0](J):=
\PF[\q_0](\nu_{\q}^{-1}(J))$ defines a chamber \grsys/ $\nu_{\q*}\bc[\q_0]$
over $\typ_g$ with vertices $\PF[\q_0](\typ_{\q_0})$. These operations are
adjoint functors, so that morphisms from $\nu_\q^*\PF\wop{\q}$ to $\PF[\q_0]$
over $\typ_{\q_0}$ correspond bijectively to morphisms from $\PF\wop{\q}$ to
$\nu_{\q*}\PF[\q_0]$. Parabolic projection is such a morphism.

\begin{thm} \label{t:pp-morphism} $\PP{\q}$ defines an incidence morphism
$\sm_\q\colon \nu_{\q}^*\PI{\g}\wop{\q} \to\PI{\q_0}$ and a chamber morphism
$\cm_\q\colon \bc\wop{\q}\to\nu_{\q*}\bc[\q_0]$.
\end{thm}

\subsection{Geometric configurations and their projections}

We end this paper by returning full circle to Section~\ref{s:eig} and the
motivation from incidence geometry. We have noted in Proposition~\ref{p:is-cs}
that any incidence system $\grph$ has an associated chamber \grsys/
$\cgph=\cC\grph$ whose chambers are the full flags of $\grph$. In fact $\cC$
is functorial and has an adjoint~\cite{Sch:hig}.

\begin{defn} Let $\cgph$ be a chamber \grsys/ over $\typ$. For $i\in\typ$,
an \emph{$i$-coresidue} of $\cgph$ is a connected component of the graph
obtained from $\cgph$ by removing all edges with label $i$.
\end{defn}
\begin{prop} For a chamber \grsys/ $\cgph$ over $\typ$, let $\cE\cgph$ be
the graph with type function $t_{\cE\cgph}\colon |\cE\cgph|\to \typ$ whose
vertices of type $i$ are the $i$-coresidues $v$ of $\cgph$ with edges $v\edge
{}w$ if and only if $v$ and $w$ contain a common chamber. Then $\cE\cgph$ is
an incidence system.
\end{prop}
This is immediate: if two $i$-coresidues contain a common chamber, they are
equal.

If $\cgph=\cC\grph$ is the chamber \grsys/ of full flags of a incidence system
$\grph$ then all the full flags in an $i$-coresidue have the same element of
type $i$; the map assigning this element to the $i$-coresidue is in fact an
incidence morphism $\cE\cC\grph\to\grph$ (the co-unit of the adjunction), and
is an isomorphism if and only if $\grph$ is:
\begin{bulletlist}
\item \emph{flag regular}, \ie, any element belongs to a full flag; and
\item \emph{residually connected}, \ie, for any $i\in\typ$, any two full flags
  containing a common element of type $i$ have the same $i$-coresidue.
\end{bulletlist}
The parabolic incidence system $\PI{\g}$ has both properties: identifying full
flags with minimal parabolics using Theorem~\ref{t:par-flag-geom}, we first
note that any maximal (proper) parabolic contains a minimal parabolic, which
implies flag regularity. Now if minimal parabolic $\pb,\pc\sub\q$ for a
maximal parabolic $\q$, then $\pb/\nil(\q)$ and $\pc/\nil(\q)$ are chambers in
$\bc[\q_0]$, which is connected. Thus $\pb,\pc$ have the same
$G\cdot\q$-coresidue, and so $\PI{\g}$ is residually connected.

If $\cgph^W$ is a homogeneous thin chamber \grsys/ over $\typ$ with structure
group $W$ generated by $\typ\into W$, then its $i$-coresidues are the orbits of
the subgroup of $W$ generated by $\typ\setminus\{i\}$. If $W=W(\g)$ then the
image of an injective chamber morphism $\cgph^{W(\g)}\to\bc$ is an apartment,
and such a labelled apartment thus determines an incidence morphism
$\grph^{W(\g)}\to\PI{\g}$.

\begin{defn} Let $\g$ be a reductive Lie algebra. A \emph{geometric
    configuration} in $\PI{\g}$ with \emph{combinatorial type} (or
  \emph{combinatorics}) $\grph$ is an incidence morphism $\grph\to\PI{\g}$.  A
  \emph{standard configuration} in $\PI{\g}$ is an incidence morphism
  $\sm\colon\grph^{W(\g)}\to\PI{\g}$ corresponding to a labelled apartment
  $\cgph^{W(\g)}\into\bc$.
\end{defn}

Standard configurations are rather special, but can be used to obtain more
general configurations by parabolic projection. Suppose then that $\q$ is
parabolic in $\g$, define $\nu_\q\colon\typ_{\q_0}\to\typ_\g$ as in the
previous section, and let $\sm\colon \grph^{W(\g)} \to\PI{\g}$ be a standard
configuration weakly opposite to $\q$, \ie, with image in
$\PI{\g}\wop{\q}$. Then $\sm$ immediately defines a configuration
$\nu_{\q}^*\sm\colon \nu_{\q}^*\grph^{W(\g)} \to \nu_{\q}^*\PI{\g}\wop{\q}$ by
restricting along $\nu_{\q}$ to elements with types in $\typ_{\q_0}$.

\begin{defn} For $\q$ parabolic in $\g$, the \emph{parabolic projection} of a
  standard configuration $\sm$ weakly opposite to $\q$ is the geometric
  configuration $\sm_\q\circ\nu_{\q}^*\sm\colon \nu_{\q}^*\grph^{W(\g)}\to
  \PI{\q_0}$.
\end{defn}

\begin{example} Recall from Example~\ref{ex:1A} the incidence system
$\grph^\cS$ of proper nonempty subsets of an $n+1=\dim V$ element set $\cS$.
An injective incidence system morphism $\grph^\cS\to\grph^{\gl(V)}$ defines a
\emph{simplex} in $\Proj(V)$---the complete configuration of all projective
subspaces spanned by subsets of $n+1$ distinct points in $\Proj(V)$; these are
the standard configurations in $\grph^{\gl(V)}$.

If $\q$ is the infinitesimal stabilizer of a point $L\in\Proj(V)$, then
parabolic projection defines an incidence system morphism
$\nu_{\q}^*\grph^{\gl(V)}\wop{\q}\to\grph^{\gl(V/L)}$, where
$\grph^{\gl(V/L)}$ is an incidence system over $\cI_{n-1}$ with the natural
inclusion of types sending $j$ to $j+1$. Since linear subspaces of
complementary dimensions are opposite, $\nu_{\q}\colon\cI_{n-1}\to\cI_n$ is
the composite $j\mapsto n-j\mapsto n-j+1 \mapsto n+1-(n-j+1)=j$. Explicitly,
if $W\leq V$ does not contain $L$ then the infinitesimal stabilizer of
$\Proj(W)$ is weakly opposite to $\q$ and is projected to (the infinitesimal
stabilizer of) $\Proj((W+L)/L)$. The image of a simplex weakly opposite to
$\q$ is thus a complete configuration of $n+1$ points in general position in
$\Proj(V/L)$.  For example, the parabolic projection of a tetrahedron in
$\RP3$ is a quadrilateral (four points and six lines) in $\RP2$.
\end{example}
\begin{example} Example~\ref{ex:4B} shows that for standard configurations
in $\grph^{\so(U,Q_U)}$, we may take the combinatorial type to be the
incidence system $\grph^{\cR,\pm}$ over $\cI_n$ of admissible subsets of
$\bigcup\cR$, where $\cR$ is a set of $n$ disjoint two element sets. This is
isomorphic to the set of all faces of an \emph{$n$-cross polytope} or
\emph{$n$-orthoplex} with the singleton subsets as vertices, or, dually, to
the set of all faces of an \emph{$n$-cube}, with the $n$-element subsets as
vertices. A standard configuration maps the vertices of the $n$-cross to
points in the quadric of isotropic $1$-dimensional subspaces of $U$, such that
the admissible subsets of vertices (corresponding to faces of the $n$-cross)
span subspaces which are isotropic, \ie, lie entirely in the quadric.

If $\q$ is the infinitesimal stabilizer of a point $L$ on the quadric, then
(the infinitesimal stabilizer of) any isotropic subspace which does not contain
$L$ and is not contained in $L^\perp$ is weakly opposite to $L$ (\ie, to
$\q$).  Its projection onto the quadric of isotropic lines in $L^\perp/L$ is
$(W\cap L^\perp+L)/L$. For example if $U=\R^{4,3}$, then $L^\perp/L$ is
isomorphic to $\R^{3,2}$, and parabolic projection maps lines and planes in a
real $5$-quadric $Q^5$ to points and lines in a $3$-quadric $Q^3$. A $3$-cross
is an octahedron, and so parabolic projection constructs configurations of
$12$ points and $8$ lines in $Q^3$. This has an interpretation in Lie circle
geometry~\cite{Cec:lsg}: points in $Q^3$ parametrize oriented circles in
$S^2\cong\R^2\cup\{\infty\}$, which have oriented contact when the line
joining the points is isotropic (\ie, lies in $Q^3$).
\end{example}

Given a geometric configuration $\nu_{\q}^*\grph^{W(\g)}\to\PI{\q_0}$, it is
natural to ask if and when it may be obtained by parabolic projection from a
standard configuration in $\PI{\g}$. This is a lifting problem which is
studied in~\cite{Nop:pp} by comparing the moduli of such configurations to the
moduli of those obtained by parabolic projection. Here a crucial role is
played by the fundamental result that all apartments in $\bc$ belong to the
apartment complex $\Ac$, hence are all conjugate by $G$.  We shall return to
these ideas in subsequent work.

\appendix
\section{Standard Lie theory background}\label{s:a}

\subsection{Semisimplicity and reducibility}\label{a:sr}

Let $A\sub \End_\F(V)$ for a finite dimensional vector space $V$ over a field
$\F$; then $(V,A)$ is \emph{simple} or \emph{irreducible} if $V$ has exactly
two invariant subspaces $0,V\neq 0$, and \emph{semisimple} or \emph{completely
  reducible} if every $A$-invariant subspace of $V$ has an $A$-invariant
complement, \ie, $V$ has a direct sum decomposition into irreducibles. These
notions depend only on the (associative) $\F$-subalgebra of $\End_\F(V)$
generated by $A$, and extend to the case that $V$ carries a representation
$\rho$ of a group, associative algebra or Lie algebra by taking $A$ to be be
the image of $\rho$ in $\End_\F(V)$.

For any field extension $\F\sub\F^c$ and any $\F$-vector space $V$, let
$V^c:=\F^c\tens_\F V$ be the induced $\F^c$-vector space, which is
functorial in $V$ (any linear map $\alpha\colon V\to W$ induces
$\alpha^c\colon V^c\to W^c$). If $\F$ is perfect (\ie, any algebraic extension
is separable) and $(V,A)$ is semisimple, then so is $(V^c,A^c)$. In
particular, if $\alpha\in \End_\F(V)$ then $\alpha$ (\ie, $(V,\{\alpha\})$) is
semisimple iff its minimal polynomial $p_\alpha$ has distinct irreducible
factors iff $\alpha^c$ is diagonalizable (\ie, $V^c$ has a basis of
eigenvectors for $\alpha^c$) in a (separable) splitting field $\F^c$ for
$p_\alpha$.

\begin{lemma}\label{l:ad-ss} If $\alpha\in\End_\F(V)$ is semisimple
and $\alpha^c$ has eigenvalues $\cS\sub \F^c$ \textup(in a separable splitting
field $\F^c$ for $p_\alpha$\textup) then $\ad\alpha\in\End_\F(\gl(V))$ is
semisimple and $(\ad\alpha)^c=\ad(\alpha^c)$ has eigenvalues $\{\lam-\mu\st
\lam,\mu\in \cS\}$.  Furthermore if, for some additive group homomorphism
$f\colon\F^c\to\F^c$, $\beta\in\End_{\F^c}(V^c)$ is scalar multiplication by
$f(\lam)$ on the $\lam$-eigenspace of $\alpha^c$ for all
$\lam\in\cS$, then $\ad\beta$ is a polynomial in $\ad\alpha^c$ with no
constant term.
\end{lemma}
\begin{proof}[\pfin{Bou:gal,Hum:lar}] Let $V_\lam:\lam\in\cS$ denote the
eigenspaces of $\alpha^c$ in $V^c$. Clearly $(\ad\alpha)^c=\ad(\alpha^c)$ has
eigenvalue $\lam-\mu$ on the subspace consisting of all $\gl(V^c)$ which map
$V_\mu$ to $V_\lam$ and all other eigenspaces to zero. These subspaces span
$\gl(V^c)$ so $\ad\alpha^c$, hence $\ad\alpha$, is semisimple. Finally,
$\ad\beta=q(\ad\alpha^c)$ where $q(\lam-\mu)=f(\lam)-f(\mu)$ for all
$\lam,\mu\in \cS$. Since $f$ is additive, such a polynomial $q$ exists by
Lagrange interpolation, and has $q(0)=0$.
\end{proof}

\subsection{Jordan decomposition}\label{a:jd}

Let $V,W$ be vector spaces over a perfect field $\F$. Any $\alpha\in\gl(V)$
has a unique (additive) \emph{Jordan decomposition} $\alpha=\alpha_s+\alpha_n$
such that $\alpha_s$ is semisimple, $\alpha_n$ is nilpotent (\ie,
${\alpha_n}^k=0$ for $k$ sufficiently large) and
$[\alpha_s,\alpha_n]=0$. Further, $\alpha_s$ and $\alpha_n$ are polynomials in
$\alpha$ with no constant term. The following properties of Jordan
decomposition, for $\alpha\in\gl(V)$ and $\beta\in\gl(W)$, are
straightforward and standard~\cite{Bou:gal}.
\begin{bulletlist}
\item If $\alpha\in \c_{\gl(V)}(A,B)$ for some $B\sub A\sub V$, then
  $\alpha_s,\alpha_n\in\c_{\gl(V)}(A,B)$.
\item If $\phi\colon V\to W$ is a linear map with
  $\phi\circ\alpha=\beta\circ\phi$ then $\phi\circ\alpha_s=\beta_s\circ\phi$
  and $\phi\circ\alpha_n=\beta_n\circ\phi$.
\item In $\gl(V\tens W)$, $(\alpha\tens 1+1\tens\beta)_s
=\alpha_s\tens 1+1\tens\beta_s$ and $(\alpha\tens 1+1\tens\beta)_n
=\alpha_n\tens 1+1\tens\beta_n$.
\item In $\gl(V^*)$, $\alpha\transp=(\alpha_s)\transp+(\alpha_n)\transp$ is
the Jordan decomposition of the transpose $\alpha\transp$.
\end{bulletlist}
Thus Jordan decomposition is preserved in tensor representations of $\gl(V)$
and the first property extends to subspaces $B\sub A$ in such a tensor
representation. In particular, if $V$ has an algebra structure, and $\alpha$
is a derivation, so are $\alpha_s$ and $\alpha_n$. Also recall that a linear
map $\rho\colon\g\to \gl(V)$ is a representation iff for any $x\in \g$,
$\ad(\rho(x))\circ\rho=\rho\circ \ad(x)\colon \g\to\gl(V)$.

\begin{prop}\label{p:Jordans} Let $\rho\colon\g\to \gl(V)$ be a representation
and $x\in \g$. Then $\ad(x)_s,\ad(x)_n\in\der(\g)$ \textup(the Lie algebra of
derivations of $\g$\textup), $\ad(\rho(x))_s= \ad(\rho(x)_s)$,
$\ad(\rho(x))_n=\ad(\rho(x)_n)$, $\ad(\rho(x)_s)\circ \rho= \rho\circ\ad(x)_s$
and $\ad(\rho(x)_n)\circ \rho=\rho\circ \ad(x)_n$.
\end{prop}
In particular, for $\rho\colon\g\to\gl(V)$ faithful and $x\in\g$ with
$\rho(x)_s=\rho(x_s)$ and $\rho(x)_n=\rho(x_n)$ for some $x_s,x_n\in\g$, it
follows that $\ad(x_s) +\ad(x_n)$ is the Jordan decomposition of $\ad(x)$.
Conversely, the following is the key to the proof of the abstract Jordan
decomposition (or preservation of Jordan decomposition) for semisimple Lie
algebras~\cite{Bou:gal,Hum:lar,Mil:lag}.

\begin{prop}\label{p:J2} Let $\rho\colon\g\to\gl(V)$ be a semisimple
representation over a field $\F$ of characteristic zero, and suppose
$x=x_s+x_n$ where $\ad(x_s)$ is semisimple, $\ad(x_n)$ is nilpotent, $x_n\in
[\g,\g]$ and $[x_s,x_n]=0$. Then $\rho(x)_s=\rho(x_s)$ and
$\rho(x)_n=\rho(x_n)$.
\end{prop}
\begin{proof} Since $\ad(x_s)$ is semisimple, $\rho(x_s)_n\in
\c_{\gl(V)}(\rho(\g))$ by Proposition~\ref{p:Jordans}, hence the inverse image
of its span is an ideal in $\g$ acting nilpotently on $V$, so $\rho(x_s)_n=0$
since $\nil_\rho(\g)=\ker\rho$; thus $\rho(x_s)$ is semisimple. Similarly,
since $\ad(x_n)$ is nilpotent, $\rho(x_n)_s\in\c_{\gl(V)}(\rho(\g))$, hence
$\rho(x_n)_s^c\in\c_{\gl(V^c)}(\rho(\g)^c)$ for any field extension
$\F\sub\F^c$.  Since $\rho(x_n)_s$ is a polynomial in $\rho(x_n)$,
$\rho(x_n)_s^c$ preserves any $\g$-invariant subspace of $V^c$; but also
$\rho(x_n)_s=\rho(x_n)-\rho(x_n)_n$, where the first term is in
$[\rho(\g),\rho(\g)]$ and the second term is nilpotent, so $\rho(x_n)_s^c$ has
vanishing trace on any such subspace. If $\F^c$ is the algebraic closure of
$\F$, $\rho(x_n)_s^c$ then vanishes on any minimal invariant subspace of
$V^c$. Since $\rho$ is semisimple, $\rho(x_n)_s=0$, \ie, $\rho(x_n)$ is
nilpotent. Now $[\rho(x_s),\rho(x_n)]=\rho([x_s,x_n])=0$, so the result
follows.
\end{proof}

\subsection{Invariant forms}\label{a:if}

An \emph{invariant form} on a Lie algebra $\g$ is a symmetric bilinear form
$(x,y)\mapsto \ip{x,y}$ such that $\ip{[z,x],y}+\ip{x,[z,y]}=0$ for all
$x,y,z\in \g$. For a subspace $\s\sub \g$, we define $\s^\perp=\{x\in \g\st
\forall y\in \s, \; \ip{x,y}=0\}$. For $\s,\t\sub \g$, $\s\sub\t^\perp$ iff
$\t\sub \s^\perp$, and we say $\s,\t$ are \emph{orthogonal}, written
$\s\perp\t$. If $\s\perp\s$, we say $\s$ is \emph{isotropic}; for any subspace
$\s$, $\s\cap\s^\perp$ is isotropic. Invariance means that $(x,y,z)\mapsto
\ip{[x,y],z}$ is a $3$-form on $\g$, and hence for any subspaces $\r,\s,\t\sub
\g$, $\r\perp [\s,\t]$ iff $\s\perp [\t,\r]$ iff $\t\perp [\r,\s]$. Thus
$\r\sub[\s,\t]^\perp$ iff $[\r,\s]\sub\t^\perp$ iff
$\r\sub\c_{\g}(\s,\t^\perp)$, and so
$\c_{\g}(\s,\t^\perp)=[\s,\t]^\perp=[\t,\s]^\perp=\c_{\g}(\t,\s^\perp)$.

\begin{prop}\label{a:pif} Let $\ip{\cdot,\cdot}$ be an invariant form on $\g$
and let $\s$ be a subspace of $\g$.
\begin{numlist}
\item $\n_{\g}(\s)\sub \c_{\g}(\s,\s^{\perp\perp})=[\s,\s^\perp]^\perp
=\n_{\g}(\s^\perp)$, $\c_{\g}(\s)\sub\c_{\g}(\s,\g^\perp)=[\s,\g]^\perp$ and
$\z(\g)\sub \c_{\g}(\g,\g^\perp)=[\g,\g]^\perp$ with equality \textup(in each
containment\textup) if $\ip{\cdot,\cdot}$ is nondegenerate.
\item If $\s\leq\g$, then $[\s,\s^\perp]\sub\s^\perp$, hence also
  $\s\cap\s^\perp\idealin\s$\textup; in particular,
  $\g^\perp$ is an isotropic ideal in $\g$ \textup(which is zero iff
  $\ip{\cdot,\cdot}$ is nondegenerate\textup).
\item If $\s\idealin\g$, then $\s^\perp\idealin\g$, hence also
  $\s\cap\s^\perp\idealin\g$, $[\s,\s^\perp]\idealin\g$ and
  $\s+\s^\perp\idealin\g$.
\item $\g\perp [\s,\s]$ if and only if $[\g,\s]\perp \s$. In particular, if
  $\ip{\cdot,\cdot}$ is nondegenerate, any isotropic ideal in $\g$ is abelian.
\end{numlist}
\end{prop}

\begin{prop}\label{a:rad-ss} Suppose $\g$ is a Lie algebra with a nondegenerate
invariant form and no nontrivial abelian ideals. Then $\g$ is semisimple.
\end{prop}
\begin{proof}[\pfin{Die:slg}] We induct on the dimension of $\g$. Let $\l$
be a minimal nontrivial ideal in $\g$; then $\l\cap\l^\perp$ is an isotropic
ideal in $\g$, hence abelian by Proposition~\ref{a:pif} (4). Since $\l$ is
nonabelian (hence simple) by assumption, $\l\cap\l^\perp=0$.  Hence $\g$ is
the orthogonal direct sum of $\l$ and $\l^\perp$, and the latter is semisimple
by induction.
\end{proof}

%
\newcommand{\noopsort}[1]{} \newcommand{\bauth}[1]{\mbox{#1}}
\newcommand{\bart}[1]{\textit{#1}} \newcommand{\bjourn}[4]{#1\ifx{}{#2}\else{
    \textbf{#2}}\fi{ (#4)}} \newcommand{\bbook}[1]{\textsl{#1}}
\newcommand{\bseries}[2]{#1\ifx{}{#2}\else{ \textbf{#2}}\fi}
\newcommand{\bpp}[1]{#1} \newcommand{\bdate}[1]{ (#1)} \def\band/{and}
\newif\ifbibtex
\ifbibtex
\nocite{}
\bibliographystyle{genbib}
\bibliography{papers}
\else

\fi
\end{document}